\documentclass[11pt]{amsart}

\usepackage[margin=3cm]{geometry}
\usepackage{amssymb,amsmath,amsthm,enumerate,hyperref}
\usepackage[latin1,utf8]{inputenc}
\usepackage{color}
\input xy
\xyoption{all}

\newcommand{\B}{\mathcal{B}}

\newcommand{\D}{\mathcal{D}}

\newcommand{\Z}{\mathbb{Z}}
\newcommand{\xto}{\xrightarrow}
\newcommand{\Hom}{\operatorname{Hom}\nolimits}
\newcommand{\onto}{\twoheadrightarrow}

\newcommand{\fmod}{\operatorname{mod}\nolimits}

\newcommand{\add}{\operatorname{add }\nolimits}
\newcommand{\modtop}{\operatorname{top }\nolimits}
\newcommand{\soc}{\operatorname{soc }\nolimits}

\newcommand{\im}{\operatorname{Im}\nolimits}
\newcommand{\End}{\operatorname{End}\nolimits}

\newcommand{\Ker}{\operatorname{Ker}\nolimits}

\newcommand{\ext}{\operatorname{Ext}\nolimits}

\newcommand{\kdim}{\operatorname{dim_\mathit{k}}\nolimits}

\newcommand{\gldim}{\operatorname{gldim}\nolimits}

\newcommand{\bsm}{\begin{smallmatrix}}
\newcommand{\esm}{\end{smallmatrix}}

\theoremstyle{plain}
 \newtheorem{thm}{Theorem}[section]
 \newtheorem{prop}[thm]{Proposition}
 \newtheorem{lemma}[thm]{Lemma}
 \newtheorem{cor}[thm]{Corollary}
 
\theoremstyle{definition}
 \newtheorem{defin}[thm]{Definition}
 
\theoremstyle{remark}
 
 \newtheorem{exm}[thm]{Example}
 \newtheorem*{notation}{Notation}
 
\date{February 28, 2018}

\begin{document}

\title[Quasi-hereditary covers of higher zigzag-algebras]{Quasi-hereditary covers of higher zigzag-algebras}
\author{Gabriele Bocca\\ University of East Anglia, Norwich, UK \\ {g.bocca@uea.ac.uk}}

\begin{abstract}
The aim of this paper is to define and study some quasi-hereditary covers for higher zigzag algebras. We will show how these algebras satisfy three different Koszul properties: they are Koszul in the classical sense, standard Koszul and Koszul with respect to the standard module $\Delta$, according with the definition given in \cite{Madsen1}. This last property gives rise to a well defined duality and the $\Delta$-Koszul dual will be computed as the path algebra of a quiver with relations.
\end{abstract}

\maketitle
\tableofcontents

\section*{Introduction}

Higher zigzag-algebras were first defined in \cite{GuoYun}, \cite{Guo2016} under the name of $n$\emph{-cubic pyramid algebras}, in relation with translation quivers appearing in higher representation theory. Independently they appeared in \cite{JG} and \cite{GI} in connection with \emph{higher preprojective algebras}. The name \emph{higher zigzag-algebras} is inspired by the fact that, under some conditions, their definition generalizes the construction of Huerfano and Khovanov \cite{hk2000}. A particularly nice set of examples is given by higher zigzag algebras of $n$-Auslander algebras $T_s^{(n)}(k)$ defined by Iyama in \cite{iyama-ct}. Iyama's ``cone'' construction is recursive, so we have examples of higher zigzag-algebras of algebras $T_s^{(n)}(k)$ with any given global dimension. Since they come from higher Auslander algebras of type $A$ quivers, these algebras have been called \emph{type A} higher zigzag-algebras. In \cite{JG} the author proved that they are symmetric and have a nice presentation as path algebras of some quivers modulo zero and commutativity relations.     
Type $A$ classical zigzag-algebras have made their appearance in many areas of representation theory of algebras and finite groups. In particular they have been studied in relation with the action of braid groups on derived categories (see for example \cite{seidelthomas}, \cite{hk2000}).
A slightly different version of these algebras appeared also in the work of many authors about \emph{quasi-hereditary algebras} (see e.g. \cite{Klamt}, \cite{KhS} and \cite{MT}). More explicitly they admit quasi-hereditary covers in the sense of Rouquier \cite{rouquier2008q} and these covers are isomorphic to quotients of ``bigger'' zigzag-algebras. Moreover these particular quasi-hereditary algebras are Koszul and, when endowed with different gradings, provide an interesting example of $\Delta$-Koszul algebras. The theory of generalized Koszul properties was first introduced in \cite{GRS}, then developed by Madsen in \cite{Madsen1} and, focusing on the quasi-hereditary case, in \cite{Madsen2}. In this article we will define some quasi-hereditary covers of higher zigzag-algebras and we will show that the results proved in \cite{Madsen2} about generalized Koszul duality apply for such algebras. Moreover we will compute explicitly their $\Delta$-Koszul dual.

This paper is organized as follows. In Section 1 we give the definition of higher zigzag-algebras as bound quiver algebras, then in Section 2 we recall the definition of quasi-hereditary algebras and we define quasi-hereditary covers of higher zigzag-algebras. A strong exact Borel subalgebra of these covers is also computed. In Section 3 we prove that our quasi-hereditary covers are Koszul and standard Koszul. Section 4 focuses on the more general notion of $\Delta$-Koszul quasi-hereditary algebra. After recalling some general facts we prove that this property is satisfied by our algebras. To conclude, in Section 5 we explicitly compute the $\Delta$-Koszul dual and we prove that that this new algebras are also Koszul in the classical sense. 

Throughout this paper $k$ will denote an algebraically closed field and $\Lambda$ a finite dimensional $k$-algebra. All modules are finitely generated right modules and the composition of morphisms $fg$ means that $g$ is applied first and then $f$. We will denote by $\mod\Lambda$ the category of finitely generated modules and by $\Hom_\Lambda(M,N)$ the vector space of $\Lambda$-module morphisms between $M$ and $N$. We will also work with (non-negatively) graded algebras: if $\Lambda=\bigoplus_{n\ge 0}\Lambda_n$ is graded with graded shift $\langle\cdot\rangle$ (that is $(M\langle j\rangle)_i=M_{i-j}$ for any graded $\Lambda$-module $M$), then $\mathsf{gr}\Lambda$ denotes the category of finitely generated graded $\Lambda$-modules and $\Hom_{\mathsf{gr}\Lambda}(M,N)$ is the vector space of graded morphisms of degree zero between $M$ and $N$. We put $\Hom_\Lambda(M,N)=\bigoplus_{i\ge 0}\Hom_{\mathsf{gr}\Lambda}(M,N\langle i\rangle)$.  If $Q=(Q_0,Q_1)$ is a quiver and $\alpha\in Q_1$ is an arrow of $Q$, $s(\alpha)$ denotes the source of $\alpha$ and $t(\alpha)$ the target: $s(\alpha)\xto{\alpha} t(\alpha)$. For two consecutive arrows $\xto{\alpha}\xto{\beta}$ their concatenation is denoted by $\alpha\beta$ so that, in the path algebra $kQ$, $\alpha\beta\neq 0$ and $\beta\alpha=0$.

\textbf{Acknowledgements} The author would like to thank Vanessa Miemietz and Joseph Grant for the guidance and help provided during the preparation of this article. He also acknowledge the financial support from the School of Science of University of East Anglia, through all the duration of his PhD, during which this article has been written.

\section{Higher zigzag-algebras} \label{algebra definition}
We start by giving a presentation for $(n+1)$-zigzag-algebras $Z^n_s$ of type $A_s$ as path algebras over quivers with relations, inspired by \cite{GI} and \cite{n-aprIyama}.

\begin{defin} \label{defin zigzag}

\begin{enumerate}
\item For $n \geq 1$ and $s \ge 1$, let $Q^{(n,s)}$ be the quiver with vertices
\[Q_0^{(n,s)} = \{x=(x_0,x_1,\cdots,x_{n}) \in\Z_{\ge0}^{n+1} \mid \sum_{i=0}^{n}x_i=s-1\}\]
and arrows
\[Q_1^{(n,s)} = \{x \xto{\alpha_i} x+f_i \mid i \in \{0, \ldots, n\}, x,x+f_i \in Q_0^{(n,s)}\},\]
where $f_i$ denotes the vector
\[ f_i = (0, \ldots,0, \overset{i-1}{-1}, \overset{i}{1}, 0, \ldots,0)\in\Z^{n+1} \]
(cyclically, that is $f_{0} = (1, 0, \ldots, 0, -1)$).
\item For $n \geq 1$ and $s \geq 1$, we define the $k$-algebra $Z^n_s$ to be the quiver algebra of $Q^{(n,s)}$ with the following relations:

For any $x\in Q_0^{(n,s)}$ and $i,j\in \{0, \ldots, n\}$ satisfying $x+f_i$, $x+f_i+f_j\in Q_0^{(n,s)}$,
\[(x \xto{\alpha_i} x+f_i \xto{\alpha_j} x+f_i+f_j)
=\begin{cases}
(x \xto{\alpha_j} x+f_j \xto{\alpha_i} x+f_i+f_j) & \text{if } x+f_j\in Q_0^{(n,s)},\\
0&\text{if } i=j.
\end{cases}\]

Note that relations can be rewritten as 
\[
\alpha_i\alpha_j=\begin{cases} \alpha_j\alpha_i & \text{if } \alpha_j\alpha_i\neq 0, \\ 0 & \text{if } i=j. \end{cases}
\]

\end{enumerate}
\end{defin}

Clearly in the previous definition any $(n+1)$-cycle is of the form
\[ x \xto{\sigma(1)} x + f_{\sigma(1)} \xto{\sigma(2)} x + f_{\sigma(1)} + f_{\sigma(2)} \xto{\sigma(3)} \cdots \xto{\sigma(n)} x + f_{\sigma(1)} + \cdots + f_{\sigma(n)} \xto{\sigma(n+1)} x, \]
for some $\sigma \in \mathfrak{S}_{n+1}$. By commutativity relations any two cycles are equal in $Z_s^n$ and each of these cycles correspond to the generator of the socle of the indecomposable projective module $P_x$. Hence, since the vertices of each $n+1$-cycle are distinct, we deduce the following:

\begin{prop}
Let $Z=Z^n_s$ be a higher zigzag-algebra and $P_x,P_y$ two indecomposable projective $Z$-modules. Then
\begin{itemize}
\item If $x\neq y$ then $\kdim\Hom_Z(P_x,P_y)\leq 1$.
\item If $x=y$ then $\kdim\Hom_Z(P_x,P_x)=2$ and the $\Hom$-space is generated by the identity and the map $\varepsilon_x:P_x\to \soc(P_x)$. 
\end{itemize}
\end{prop}

\section{Quasi-hereditary covers of higher zigzag-algebras}

In this section we will describe a construction of \emph{quasi-hereditary covers} (``qh-covers'' for short) for higher zigzag-algebras. They will provide a good example for a more general Koszul theory (\cite{Madsen1}) for which it is possible to compute explicitly the ``Koszul dual algebra''.

\subsection{Quasi-hereditary algebras}

Let $\Lambda$ be a finitely generated $k$-algebra and suppose that we can label the set of simple $\Lambda$-modules by a partially ordered set $(I,\leq)$. $I$ is often called a \emph{set of weights} for $\Lambda$.
\begin{defin}\label{qhdefin}
For every $i\in I$ the \emph{standard module} $\Delta_i$ is the largest quotient of the indecomposable projective module $P_i$ having no simple composition factor $S_j$ for $i<j$. Dually the \emph{costandard module} $\nabla_i$ is the largest submodule of the indecomposable injective $I_i$ having no composition factor $S_j$ for $i<j$. Let $\Delta=\bigoplus_{i\in I}\Delta_i$ and $\nabla=\bigoplus_{i\in I}\nabla_i$.

The algebra $\Lambda$ is said to be \emph{quasi-hereditary} if the following hold:
\begin{enumerate}[(i)]
\item $\End_\Lambda(\Delta_i)\cong k$.
\item $\Lambda_\Lambda$ admits a $\Delta$-filtration, that is there exists a filtration
\[0=M_0\subset M_1\subset\ldots\subset M_{n-1}\subset M_n=\Lambda \]
such that $M_k/ M_{k-1}$ is a standard module for every $k=1,\ldots,n$.
\end{enumerate}
\end{defin} 

For the classical definition and basic properties see for instance \cite{CPS} and \cite{DR}

The definition of quasi-hereditary algebra is ``self-dual'' in the sense that $\Lambda$ is $\Delta$-filtered if and only if the module $D(\Lambda)$ is $\nabla$-filtered (\cite{DR2})

Let us recall some basic properties of quasi-hereditary algebras that will be needed in the sequel (for a proof we refer to \cite{DR2}):

\begin{prop}
Let $\Lambda$ be a qh-algebra, then:
\begin{enumerate}
\item $\Lambda$ has finite global dimension.
\item For all $i,j\in I$ and $n>0$, $\ext^n_\Lambda(\Delta_i,\nabla_j)=0$.
\item For all $i,j\in I$ we have 
\[ \Hom_\Lambda(\Delta_i,\nabla_j)\cong\begin{cases}0 \mbox{ if }i\neq j \\ k \mbox{ if }i=j .\end{cases}\]
\end{enumerate}
\end{prop}

Now we give the following definition of \emph{quasi-hereditary cover} for a $k$-algebra $\Lambda$ due to Rouquier (\cite{rouquier2008q}):

\begin{defin}
Let $\Lambda$ be a quasi-hereditary algebra and $M$ a finitely generated projective $\Lambda$-module. The pair $(\Lambda,M)$ is a \emph{quasi-hereditary cover} (or \emph{qh-cover}) for $A=\End_\Lambda(M)$ if the restriction of the functor 
\[ F=\Hom_\Lambda(M,-):\fmod \Lambda\to\fmod A\]
to the subcategory of projective $\Lambda$-module is fully faithful.
\end{defin}

%To that purpose recall the following theorem (Theorem 1.1 of \cite{} \textcolor{red}{finiteness rep dim di Iyama}):

%\begin{thm}Let $A$ be an artin algebra. Then any module $M\in \fmod A$ is the direct summand of a module $N\in \fmod A$ such that $\End_A(N)$ is a quasi-hereditary algebra. \end{thm}
\subsection{Covers for higher zigzag-algebras}

Let $Z^+=Z^n_{s+1}$ be the $(n+1)$-zigzag-algebra of type $A_{s+1}$ and consider the presentation of $Z^+$ as the path algebra of the quiver \\$(Q_0^{(n,s+1)},Q_1^{(n,s+1)})$ with relations as given in Section \ref{algebra definition} (Definition \ref{defin zigzag}). Denote by $I$ the set of vertices. In the quiver we have arrows labelled by
\[\alpha_k:(x_0,\ldots,x_{k-1},x_{k},\ldots,x_{n})\to (x_0,\ldots,x_{k-1}-1,x_{k}+1,\ldots,x_{n})\]
for $k=1,\ldots n$ and $\alpha_0$ is defined cyclically. Denote by $Z=Z^n_s$ the $(n+1)$-zigzag-algebra of type $A_s$; we can select a subset $J\subset I$ such that $Z\cong \End_{Z^+}(\bigoplus_{y\in J}P^+_y)$ where $P_y^+$ is the projective cover of the simple $Z^+$-module associated to the vertex $y$ (this is \cite{JG}, Proposition 4.4). To be precise $J=\{(x_0,x_1,\ldots,x_n)\in I\mid x_0\neq 0\}$. Put $P^+_J=\bigoplus_{y\in J}P^+_y$.

For $K=I\setminus J$, let
\[\Gamma(Z)=Z^+/ \left( e_z\alpha_0\alpha_1 | z\in K \right) \]
so that we have a surjective morphism of $k$-algebras: $ Z^+\to\Gamma(Z)$
that induces a fully faithful embedding: $\fmod\Gamma(Z)\to\fmod Z^+$. 
Note that the vertices in $K$ are precisely the ones of the kind $(0,x_1,\ldots,x_{n})$, so they are not the target of any arrow $\alpha_0$. Viceversa, vertices in $J$ are always the target of some arrow $\alpha_0$. In the following we put $\Gamma=\Gamma(Z)$.

\begin{lemma}\label{zigzag projectives}
The action of $Z^+$ on the module $P_J^+$ factors over $\Gamma$.
\end{lemma}
\begin{proof}
It is enough to show that, for any $y\in J$, any element $e^+_z\alpha_0\alpha_1$ annihilates $e^+_yZ^+=P^+_y$, where $e^+_z$ (resp. $e^+_y$) is the primitive idempotent in $Z^+$ associated to the vertex $z\in K$ (resp. $y\in J$). Let $e^+_y\alpha_i\cdots\alpha_j e_z^+$ be a path from $y$ to $z$ corresponding to a generator of $P^+_y$, such that we can non-trivially multiply it on the right by $e_z^+\alpha_0\alpha_1$. By the relations $\alpha_i\alpha_j=\alpha_j\alpha_i$ of $Z^+$, such a path is equivalent to $e^+_y\alpha_l\cdots\alpha_1 e_z^+$. Then, again using the commutativity relations, we have $e^+_y\alpha_l\cdots\alpha_1 e_z^+\alpha_0\alpha_1\sim e^+_y\alpha_l\cdots\alpha_0\alpha_1\alpha_1=0$ 
\end{proof}

%we want to factor $Z^+$ by the ideal generated by all the paths that start at a vertex $z\in K$ and end in $z'\in K$ passing through a vertex $y\in J$: \[\widetilde{Z}=Z^+/\bigoplus_{z,z'\in K,y\in J}e_zZ^+e_yZ^+e_{z'} \]
%If we call $R=\bigoplus_{z,z'\in K,y\in J}e_zZ^+e_yZ^+e_{z'}$, we have that $P^+_J\cap R=0$ and then \[\End_{\widetilde{Z}}(P_J^+)\cong\End_{Z^+}(P_J^+)\cong Z\]
%\begin{rmk}\label{rmk-K} Note that the vertices in $K$ are precisely the ones of the kind \\ $(0,x_1,\ldots,x_{n})$, so they are not the target of any arrow $\alpha_0$. Viceversa, vertices in $J$ are always the target of some arrow $\alpha_0$. Moreover, if $z\in K$ and $y\in J$, then every arrow $\alpha_l:z\to y$ is of  type $\alpha_0$ and, by the description of the quiver $(Q_0^{(n,s+1)},Q_1^{(n,s+1)})$ (see Definition \ref{defin zigzag}), we can rewrite the new relations of the algebra $\Gamma$ as:\[ \bigoplus_{z,z'\in K,y\in J}e_zZ^+e_yZ^+e_{z'}= \left( e_z\alpha_0\alpha_1 | z\in K \right) \] \end{rmk}
As a corollary we have that:
\[\End_{\Gamma}(P_J^+)\cong\End_{Z^+}(P_J^+)\cong Z\]
where the first isomorphism comes from the fully faithful embedding $\fmod\Gamma\to\fmod Z^+$.
So $P^+_J$ is also a left $Z$-module and there is an adjunction 
\[G=-\otimes_Z P^+_J:\fmod Z\rightleftharpoons\fmod\Gamma:F=\Hom_{\Gamma}(P_J^+,-)\]
such that $GF\cong id$ when restricted to $\add P_J^+$. 

\begin{lemma}
Let $P_y$ be an indecomposable projective $\Gamma$-module such that $y\in J$. Then $P_y=I_y$ is also injective.
\end{lemma}

\begin{proof}
Note that for any $y\in J$, every path in the quiver of $\Gamma$ is contained in a minimal cycle  at the vertex $y$ and this cycle is non-zero in $\Gamma$ since $P_y$ is also a module over $Z^+$. Denote by $\varepsilon_y$ the non-zero element in $\Gamma$ corresponding to such cycle (note that this is well defined since by commutativity relations all the minimal cycles at $y$ are equivalent). Let $\pi\in e_y\Gamma=P_y$ be a path starting at $y$ and $\pi'$ be its complementary path in a minimal cycle: $\varepsilon_y=e_y\pi\pi'e_y$.  Then the map $\pi \mapsto D(\pi')$ extends to a morphism of $\Gamma$-modules and it gives an isomorphism $P_y=e_y\Gamma\xto{\cong} I_y=D(\Gamma e_y)$. 
\end{proof}

We can define a partial order on $I$ by putting $x<y$ if and only if there is a path $\xymatrix{x\ar@{-->}[r]^\pi & y}$ in the quiver of $Z$, such that $\pi$ does not involve any arrow $\alpha_0$ (recall the presentation of $Z$ as in Definition \ref{defin zigzag}).

\begin{lemma}\label{stand modules}
The indecomposable standard modules of $\Gamma$ are either simple or uniserial with radical length two. Whenever we have an arrow $\alpha_0:x\to y$ then the standard module $\Delta_x$ has simple top $S_x$ and simple socle $S_y$.

Dually indecomposable costandard modules are either $\nabla_x=e_x\alpha_0\Gamma\subseteq I_x$, if $x\in J$, or $\nabla_x=I_x$, if $x\in K$.
\end{lemma}
\begin{proof}
 By the construction of the quiver of $Z^+$, there exists an arrow $\alpha_l:y\to x$ for $x<y$ only if $l=0$. Hence $\Delta_x$ is non simple if and only if there is an arrow $x\xto{\alpha_0}y$; $S_y$ is the only composition factor in a radical filtration of $P_x$ such that $y<x$, since $\alpha_0\alpha_0=0$. Then standard modules are uniserial with radical length at most 2. 
 
Dually we show that, for any $y\in J$, the costandard module $\nabla_y$ is given by $e_y\alpha_0\Gamma\subseteq I_y\cong P_y$. First $e_y\alpha_0\Gamma\subseteq\nabla_y$ because $\alpha_0$ appears only once in any path starting at $y$, so any composition factor $S_z$ of $e_y\alpha_0\Gamma$ is such that $z\leq y$. Also $\nabla_y\subseteq e_y\alpha_0\Gamma$: if not we could find an element $e_y\alpha_j\cdots\alpha_k e_z$ in $\nabla_y$ with $j,\ldots,k\neq 0$ and $S_z$ would be a composition factor of $\nabla_y$ such that $z>y$, a contraddiction.

To conclude, if $z\in K$ then $\nabla_z=I_z$. Indeed, using commutativity relations and $e_z\alpha_0\alpha_1=0$, any path ending in $z$ involving an arrow $\alpha_0$ is zero in $\Gamma$. So any composition factor $S_z$ of $I_y$ satisfies $z\leq y$.\end{proof}

\begin{exm}\label{exm qhcover}
Consider the $2$-zigzag algebra $Z=Z^{(2,3)}$, with labels on the arrows accordingly to Definition \ref{defin zigzag} (for simplicity we label the vertices using natural numbers instead of tuples):

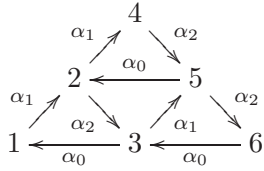
\[
\xymatrix@=1em{ & & 4 \ar[rd]^{\alpha_2} & & \\
          & 2 \ar[ur]^{\alpha_1}\ar[dr]_{\alpha_2} & & 5 \ar[dr]^{\alpha_2}\ar[ll]_{\alpha_0} & \\
          1 \ar[ur]^{\alpha_1} & & 3 \ar[ur]_{\alpha_1}\ar[ll]^{\alpha_0} & & 6 \ar[ll]^{\alpha_0}  }
\]
Then the quiver of $\Gamma$ is
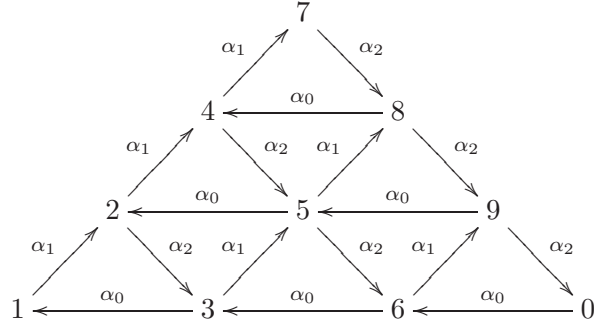
\[
\xymatrix{ & & & 7\ar[rd]^{\alpha_2} & & & \\ 
               & & 4 \ar[rd]^{\alpha_2} \ar[ur]^{\alpha_1} & & 8\ar[ll]_{\alpha_0}\ar[rd]^{\alpha_2} & &  \\
          & 2 \ar[ur]^{\alpha_1}\ar[dr]^{\alpha_2} & & 5\ar[ur]^{\alpha_1} \ar[dr]^{\alpha_2}\ar[ll]_{\alpha_0} & & 9\ar[ll]_{\alpha_0}\ar[rd]^{\alpha_2} &  \\
          1 \ar[ur]^{\alpha_1} & & 3 \ar[ur]^{\alpha_1}\ar[ll]_{\alpha_0} & & 6 \ar[ur]^{\alpha_1} \ar[ll]_{\alpha_0} & & 0\ar[ll]_{\alpha_0}  }
\]
and the ideal of relations is generated by the usual relations of $Z^{(2,4)}$ with moreover $e_8\alpha_0\alpha_1=e_9\alpha_0\alpha_1=e_0\alpha_0\alpha_1 =0$. Note that in this example we have $J=\{1,2,\ldots,6 \}$ and $K=\{7,8,9,0\}$.

The Hasse quiver of the partial order on $I$ is the following:
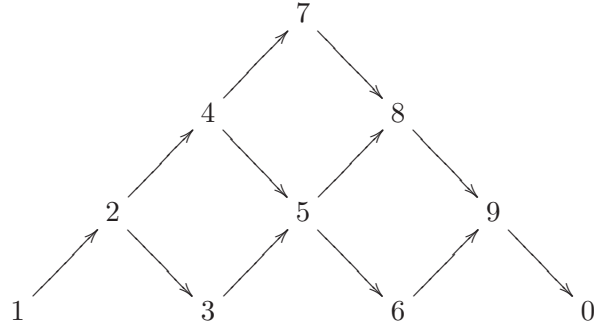
\[
\xymatrix{ & & & 7\ar[rd] & & & \\ 
               & & 4 \ar[rd] \ar[ur] & & 8\ar[rd] & &  \\
          & 2 \ar[ur]\ar[dr] & & 5\ar[ur] \ar[dr] & & 9\ar[rd] &  \\
          1 \ar[ur] & & 3 \ar[ur] & & 6 \ar[ur]  & & 0  }
\]
so that $1<2<3<5\ldots$, $1<2<4<5\ldots$, $1<2<4<7\ldots$, etc.
\end{exm}

\begin{prop} \label{quasi hered}
The algebra $\Gamma=\Gamma(Z)$ is a quasi-hereditary cover of $Z$.
\end{prop}
\begin{proof}
First we show that $\Gamma$ is quasi-hereditary. By Lemma \ref{stand modules} $\End_\Gamma (\Delta_x)\cong k$ for every $x\in I$ so we only need to show that $\Gamma$ is $\Delta$-filtered.

We will prove that every indecomposable injective module is $\nabla$-filtered. 
Let $x,y,z\in I$ such that there are arrows $z\xto{\alpha_0}y\xto{\alpha_0}x$ and consider the corresponding morphisms between indecomposable injective modules $I_x\xto{\alpha_0\cdot} I_y\xto{\alpha_0\cdot}I_z$. Since the vertices $x$ and $y$ must belong to $J$ we have that $I_x=P_x=e_x\Gamma$ and $I_y=P_y=e_y\Gamma$. Hence $\im(I_x\xto{\alpha_0\cdot} I_y)=e_y\alpha_0\Gamma=\nabla_y$ and, since by the relations we have $\Ker(I_y\xto{\alpha_0\cdot} I_z)=e_y\alpha_0\Gamma$, the sequence $I_x\xto{\alpha_0\cdot} I_y\xto{\alpha_0\cdot}I_z$ is exact. If $z\in K$, then $\nabla_z=I_z$ and $I_y\xto{\alpha_0\cdot}I_z$ is an epimorphism since the image is the costandard module $\nabla_z$. Then for every $x\in I$ such that $z\xto{\alpha_0}y'\xto{\alpha_0}\cdots\xto{\alpha_0}y\xto{\alpha_0}x$ is a subquiver of the quiver of $\Gamma$ with $z\in K$, an injective resolution of the costandard module $\nabla_x$ is given by
\[ 0\to \nabla_x\to I_x \xto{\alpha_0\cdot} I_y \xto{\alpha_0\cdot} \cdots \xto{\alpha_0\cdot} I_{y'} \xto{\alpha_0\cdot} I_z\to 0. \]
This means that for any $x\in I$, either $x$ is in $K$ and $I_x=\nabla_x$ or there exists a short exact sequence 
\[0\to \nabla_x\to I_x \to \nabla_y \to 0\]
so that $I_x$ is $\nabla$-filtered.

The functor 
\[F=\Hom_{\Gamma(Z)}(P_J^+,-):\fmod\Gamma(Z)\to\fmod Z\]
is clearly full. To prove that it is faithful let $P_x$ and $P_z$ be two indecomposable projective $\Gamma$-modules and $\pi:P_x\to P_z$ a morphism between them; hence $\pi$ is given by an equivalence class in $\Gamma$ of a path from $z$ to $x$. The image of $\pi$ through $F$ is $F\pi:\Hom_\Gamma(P_J^+,P_x)\to \Hom_\Gamma(P_J^+,P_z)$ and it is given by composing with $\pi$ on the left. We consider two cases:
\begin{itemize}
\item[-] If $x\in J$ then $id_{P_x}\in\Hom_\Gamma(P_J^+,P_x)$. Suppose $F\pi=0$, then $F\pi(id_{P_x})=\pi\cdot id_{P_x}=0$ implies $\pi=0$. 
\item[-]  If $x\in K$, suppose $\pi\neq 0$ and let $\pi'$ be such that $\pi\pi'$ is a maximal path contained in a minimal cycle based at $z$ ending in a vertex $y\in J$. Obviously in $\pi\pi'$ any arrow $\alpha_i$ can appear at most once because it is part of a minimal cycle. Moreover it does not contain any subpath $\alpha_0\alpha_1$: if this was the case, another $\alpha_0$ would have to appear in $\pi\pi'$ for this path to end in $J$. Then $F\pi(\pi')=\pi\pi'$ is a non-zero morphism. 
\end{itemize}  
We have proved that $\pi\neq 0$ implies $F\pi\neq0$ so the functor $F$ is faithful and the proof is complete. \end{proof}
 
\begin{notation}
Despite the non-uniqueness of quasi-hereditary covers, we will from now on refer to $\Gamma=\Gamma(Z)$ as the qh-cover defined in this section.
\end{notation}

%Consider $y\in J$ and $x\in I$ such that there exists $x\xto{\alpha_0}y$. Then $P_y=I_y=e_y\Gamma$ is projective and injective and the kernel $C$ of the morphism $I_y\xto{\alpha_0\cdot} I_x$ is generated as a vector space by all the non-zero paths starting at $y$ non involving arrows $\alpha_0$. This means that $C$ has simple socle $S=\langle e_y\alpha_1\cdots\alpha_n e_x\rangle\cong S_x$ since the arrow $x\xto{\alpha_0}y$ is the only arrow $\alpha_0$ whose target is $y$. Hence $C$ is a submodule of $I_x$ whose composition factors are isomorphic to $S_{x'}$ for $x\ge x'$. The composition $I_y\to C\to I_x$ coincides with the unique map $I_y\xto{\alpha_0\cdot}I_x$ so $C$ is the largest submodule of $I_x$ satisfying the condition on composition factors and this means that $C\cong \nabla_x$. We have a short exact sequence:\[0\to \nabla_y\xto{\alpha_0\cdot} I_y \to \nabla_x \to 0\] that shows that $P_y=I_y$ is $\nabla$-filtered, so $\Gamma$ is quasi-hereditary. \end{proof}\\

%\begin{rmk} From the proof of Proposition \ref{quasi hered} we deduce that an injective coresolution of $\nabla_x$ is given by \[ 0\to \nabla_x\to I_x \xto{\alpha_0\cdot} I_y \xto{\alpha_0\cdot} \cdots \xto{\alpha_0\cdot} I_{y'} \xto{\alpha_0\cdot} I_z\to 0 \] for any $x\in I$, $y,y'\in J$ and $z\in K$. \end{rmk}

\subsection{Strong exact Borel subalgebras}
An important notion related to quasi-hereditary algebras are \textit{exact Borel subalgebras}, first defined by K\"{o}nig in \cite{Koe95}. The aim was, given a quasi-hereditary algebra $\Lambda$, to find a subalgebra of $\Lambda$ that somehow encodes the information about the standard filtration of projective $\Lambda$-modules.

\begin{defin}
Let $\Lambda$ be a quasi-hereditary algebra, $I$ an index set for the simple $\Lambda$-modules and $\leq$ the partial order on $I$. A subalgebra $B$ of $\Lambda$ is called an \emph{exact Borel subalgebra} if and only if the following three conditions are satisfied:
\begin{itemize}
\item The algebra $B$ has the same partially ordered set of indices of simple modules as $\Lambda$ and $B$ is directed, i.e. it is quasi-hereditary with simple standard modules;
\item The functor $\Lambda\otimes_B-$ is exact;
\item The functor $\Lambda\otimes_B-$ sends simple $B$-modules to standard $\Lambda$-modules and this correspondence preserves the indices.
\end{itemize}
An exact Borel subalgebra $B$ of $\Lambda$ is called \emph{strong} if $\Lambda$ has a maximal semisimple subalgebra that is also a maximal semisimple subalgebra of $B$.
\end{defin}

In general it is not true that every quasi-hereditary algebra has an exact Borel subalgebra. Nevertheless it has been proved in \cite{KKO13} that every quasi-hereditary algebra is Morita equivalent to a quasi-hereditary algebra that has an exact Borel subalgebra.

A useful tool to determine the existence of an exact Borel subalgebra is the following:
\begin{thm}\cite[Theorem A]{Koe95}\label{borel exist}
Let $\Lambda$ be a quasi-hereditary algebra and $B$ a subalgebra of $\Lambda$. Suppose that the index set of simple $\Lambda$-modules $I$ is in bijection with the index set of simple $B$-modules so that we have an induced partial order on simple $B$-modules. Then $B$ is an exact Borel subalgebra of $\Lambda$ if and only if $B$ with the partial order defined above is directed and satisfies the following condition:
\begin{itemize}
\item For each $x\in I$, the restriction from $\Lambda$-modules to $B$-modules gives an isomorphism of costandard modules: $(\nabla_x)_\Lambda\cong(\nabla_x)_B$.
\end{itemize} 
\end{thm}

\begin{prop}
Qh-covers of higher zigzag-algebras have strong exact Borel subalgebras.
\end{prop}
\begin{proof}
Let $Q$ be the quiver of $\Gamma$ and $Q'$ the subquiver of $Q$ with the same set of vertices $Q'_0=Q_0$ and arrows $Q'_1=\{\alpha_i\in Q_1 \mid i\neq 0\}$. The path algebra of $Q'$ bound by the ideal of relations $R'=\{\alpha_i\alpha_i \mid i\in I\}\cup\{ \alpha_i\alpha_j=\alpha_j\alpha_i \mid i,j\in I\text{ and }\alpha_i\alpha_j\neq 0\}$ is a subalgebra of $\Gamma$ and we will denote it by $B$. Since to obtain $Q'$ from $Q$ we removed exactly the arrows $\alpha_0$, we can identify the vertices of the two quivers and label them with the same set $I$. Note that, by comparing their presentations, the algebras $B$ and the quadratic dual of the higher Auslander algebra $\Lambda^{(n,s+1)}$ of type $A$ are isomorphic. The partial order on $I$ that we gave before is still well defined in the quiver $Q'$ since we have not removed any arrow $\alpha_i$ for $i\neq 0$. Then $B$ has simple standard modules and, for every projective indecomposable $B$-module $P$, any composition series of $P$ gives a $\Delta$-filtration. This means that $B$ is a quasi-hereditary algebra with simple standard modules. The restriction functor $\iota:\fmod\Gamma\to\fmod B$ sends costandard modules to costandard modules, since they are generated by paths not involving arrows $\alpha_0$, and the indices are preserved. Then the condition of Theorem \ref{borel exist} is satisfied and $B$ is a strong exact Borel subalgebra of $\Gamma$. \end{proof}

\section{Koszulity and standard Koszulity}

\subsection{Definitions and classic results}
Before going on, we recall the definition of Koszul algebras and some generalizations of this notion. We refer to \cite{BSGkoszul} for the definition of (classical) Koszul algebras and their basic properties. The definitions of \emph{standard Koszul} and \emph{T-Koszul} algebras are taken from \cite{ADLquasi} and \cite{Madsen2}, \cite{Madsen1} respectively.

Koszul algebras are graded algebras so let us consider a finite dimensional graded $k$-algebra $\Lambda=\bigoplus_{i=0}^t\Lambda_i$, such that $\Lambda_0\simeq k^s$ is semisimple.
Denote by $\mathsf{gr}\Lambda$ the category of finitely generated graded $\Lambda$-modules and for any graded module $M$ let $M\langle j\rangle$ be the \emph{jth graded shift} of $M$, with graded parts $(M\langle j\rangle)_i=M_{i-j}$.

\begin{defin} $\Lambda$ is called a \emph{Koszul algebra} if $\ext_{\mathsf{gr}\Lambda}^i(\Lambda_0,\Lambda_0\langle j\rangle)=0$ whenever $i\neq j$.
Suppose moreover that $\Lambda$ is quasi-hereditary; then $\Lambda$ is \emph{standard Koszul} if $\ext_{\mathsf{gr}\Lambda}^i(\Delta,\Lambda_0\langle j\rangle)=0$ whenever $i\neq j$. 
\end{defin}

\begin{defin}
A finitely genereted graded $\Lambda$-module is called \emph{linear} if $M$ admits a graded projective resolution 
\[ \cdots\to P^2\to P^1\to P^0 \to M \to 0 \]
such that $P^i$ is generated by its component of degree $i$, for any $i\geq 0$.
\end{defin}

The following fact is a very useful characterization of Koszul modules:

\begin{prop}\cite[Proposition 1.14.2]{BSGkoszul}
Let $M$ be a finitely generated graded $\Lambda$-module. The following are equivalent:
\begin{enumerate}
\item $M$ is linear.
\item $\ext_{\mathsf{gr}\Lambda}^i(M,\Lambda_0\langle j\rangle)=0$ unless $i=j$.
\end{enumerate}
\end{prop}
Then $\Lambda$ is Koszul if and only if simple $\Lambda$-modules are linear and similarly, if $\Lambda$ is quasi-hereditary, it is standard Koszul if and only if the module $\Delta$ is linear.

Let $\Lambda=kQ/I$ be the path algebra of the quiver $Q$ with relations given by an homogeneous admissible ideal $I\subseteq kQ_2$, where $Q_i$ denotes the set of paths of length $i$. Let $\mathcal{B}=\bigcup_{i\geq 0}\mathcal{B}_i$ be a basis of $\Lambda$ consisting of paths such that $\mathcal{B}_0=Q_0$, $\mathcal{B}_1=Q_1$ and $\mathcal{B}_i\subseteq Q_i$. Suppose moreover that we can define a total order $<$ on $Q_1$ that we extend to each $\mathcal{B}_i$ lexicographically and then to the union $\mathcal{B}_+=\bigcup_{i> 0}\mathcal{B}_i$, by refining the degree order.

\begin{defin}
The pair $(\B,<)$ is a \emph{PBW basis} for $\Lambda$ if the following hold:
\begin{itemize}
\item if $p$ and $q$ are paths in $\B$ then either $pq$ is in $\B$ or it is a linear combination of elements $r\in\B$ with $r<pq$.
\item a path $\pi=\alpha_1\alpha_2\cdots\alpha_i$ of length $i\geq 3$ is in $\B$ if and only if for each $1\leq j\leq i-1$ the paths $\alpha_1\cdots\alpha_j$ and $\alpha_{j+1}\cdots\alpha_i$ are in $\B$.
\end{itemize}
\end{defin}

The following fact can be found in \cite[Theorem 2.18]{JG}, generalizing Theorem 5.2 in \cite{Priddy}:

\begin{thm}
If $\Lambda$ has a PBW basis then it is Koszul.
\end{thm}

We recall now some results about Koszul duality from \cite{BSGkoszul}. First of all we have the following:
\begin{prop}\cite[Corollary 2.3.3]{BSGkoszul}
Any Koszul algebra $\Lambda$ is quadratic, i.e. the ideal of relations is generated in degree two.
\end{prop}

We denote by $\Lambda^!$ the quadratic dual of $\Lambda$ and $E(\Lambda)=\ext_\Lambda^*(\Lambda_0,\Lambda_0)$ the graded algebra of self-extensions of $\Lambda_0$. 

\begin{thm}\cite[Theorems 2.10.1, 2.10.2]{BSGkoszul}
Let $\Lambda$ be a Koszul algebra, then $E(\Lambda)\cong(\Lambda^!)^{op}$ and $E(E(\Lambda))\cong\Lambda$ canonically.
\end{thm}

This ``duality'' between $\Lambda$ and $E(\Lambda)$ gives raise to an equivalence of triangulated categories as explained in the following Theorem:

\begin{thm}\cite[Theorems 2.12.5, 2.12.6]{BSGkoszul}
There exists an equivalence of triangulated categories
\[ \mathcal{K}:\mathcal{D}^b(\Lambda)\to \mathcal{D}^b(\Lambda^!)\]
between the (graded) bounded derived category of $\Lambda$ and the one of $\Lambda^!$ such that:
\begin{enumerate}[(a)]
\item $\mathcal{K}(M\langle n\rangle)=(\mathcal{K}M)[-n]\langle -n\rangle$ canonically for any $M\in \mathcal{D}^b(\Lambda)$.
\item Let $S_i=e_i\Lambda_0$ be the simple $\Lambda$-module associated to the vertex $i$, $I_i$ its injective envelope and $P_i=e_i\Lambda^!$ the projective cover of the simple $\Lambda^!$-module $e_i\Lambda^!_0=T_i$, then $\mathcal{K}(S_i)=P_i$ and $\mathcal{K}(I_i)=T_i$.
\end{enumerate}
The functor $\mathcal{K}$ is called the ``Koszul duality functor''.
\end{thm}

For the construction of the functor $\mathcal{K}$ we refer to Theorem 2.12.1 of \cite{BSGkoszul}.

In Sections 2.13 and 2.14 of \cite{BSGkoszul} the authors characterized the class of Koszul modules of $\Lambda^!$ (see also the Remark following Theorem 2.12.5 in the same paper). Let \[\mathsf{gr}\Lambda^\uparrow=\{M\in\mathsf{gr}\Lambda \mid M_j=0 \text{ for } j\ll 0 \}\] and \[\mathsf{gr}\Lambda^\downarrow=\{M\in\mathsf{gr}\Lambda \mid M_j=0 \text{ for } j\gg 0 \}\]
be the subcategories of $\mathsf{gr}\Lambda$ consisting of modules whose degree is bounded below and above respectively. 

\begin{prop}\cite[Corollary 2.13.3]{BSGkoszul}\label{koszul modules}
The class of linear modules of $\Lambda^!$ consists precisely of $\mathsf{gr}(\Lambda^!)^\uparrow\cap \mathcal{K}(\mathsf{gr}\Lambda^\downarrow)$.
\end{prop}

\subsection{Koszulity of qh-covers}

Note that qh-covers of higher zigzag-algebras are quadratic and we can define an order on the arrows such that $e_x\alpha_i<e_x\alpha_{i+1}$ for $i=1,\ldots,n-1$. Moreover we set $e_x\alpha_i<e_x\alpha_0$ for every $i\neq 0$. 

\begin{prop}
If $Z$ is a higher zigzag-algebra, then its qh-cover $\Gamma$ is a Koszul algebra.
\end{prop}

\begin{proof}
We want to show that $\Gamma$ has a PBW basis. Since we already have an order on the arrows, we need to show that we can extend this order lexicographically to paths of any length. Remember that if $Z=Z^n_s$, then $\Gamma$ is a quotient of $Z^n_{s+1}$, so we can label the vertices of the underlying quiver by $x=(x_0,x_1,\ldots,x_{n+1})$ where $\sum_ix_i=s$. If we want to extend our order we have to prove that, for every $i<j$, if $e_x\alpha_j\alpha_i\neq 0$ then $e_x\alpha_j\alpha_i=e_x\alpha_i\alpha_j$. If this is the case then
\[ \B=\{ e_x\alpha_{i_1}\cdots\alpha_{i_s} \mid x\in Q_0^{(n+1,s)}, {i_1}<{i_2},\cdots<{i_s} \} \]  
is a PBW basis for $\Gamma$.

Now suppose we have $i<j$ and $e_x\alpha_j\alpha_i\neq 0$:

\[ \xymatrix@=1.5em{ x=(\ldots,x_{i-1},x_{i},\ldots,x_{j-1},x_{j},\ldots) \ar[d]^{\alpha_j} \\ (\ldots,x_{i-1},x_{i},\ldots,x_{j-1}-1,x_{j}+1,\ldots) \ar[d]^{\alpha_i} \\
(\ldots,x_{i-1}-1,x_{i}+1,\ldots,x_{j-1}-1,x_{j}+1,\ldots)
} \]
It is clear that the composition $e_x\alpha_i\alpha_j$ always exists unless $j=n+1$, $i=0$ and $x_i=x_0=0$. But in this last situation we have that $x\in K$ and $e_x\alpha_j\alpha_i=e_x\alpha_0\alpha_1=0$ in $\Gamma$. \end{proof}

It is known that standard Koszul algebras are Koszul in the classical sense; this follows from a characterisation of standard Koszul algebras that is Theorem 1.4 in \cite{ADLquasi}. 

\begin{thm}
The quasi-hereditary algebra $\Gamma$ is standard Koszul.
\end{thm}

\proof
Every standard module $\Delta_x$ is either simple or the extension of two simple modules
\[ 0\to S_y\langle 1\rangle \to\Delta_x\to S_x\to 0\]
such that there exists an arrow $x\xto{\alpha_0}y$. Let $P^\bullet(x)$ and $P^\bullet(y)$ be linear projective resolutions of $S_x$ and $S_y$ respectively (their existence is provided by the Koszulity of $\Gamma$). Then in $\mathcal{D}^b(\Gamma)$ there is a triangle:
\[ \Delta_x\to P^\bullet(x) \to P^\bullet(S_y\langle 1\rangle)[1]\xto{+} \]

Now consider the Koszul duality functor
\[ \mathcal{K}:\mathcal{D}^b(\Gamma^!)\to \mathcal{D}^b(\Gamma)\]
and denote its quasi-inverse by $\mathcal{K}^{-1}$. Applying $\mathcal{K}^{-1}$ to the previous triangle we obtain a triangle in $\mathcal{D}^b(\Gamma^!)$:
\[ C\to \mathcal{K}^{-1}(P^\bullet(x)) \to \mathcal{K}^{-1}(P^\bullet(S_y\langle 1\rangle)[1])\xto{+} \]
where $\mathcal{K}(C)\cong\Delta_x$, $\mathcal{K}^{-1}(P^\bullet(x))\cong\mathcal{K}^{-1}(S_x)\cong I_x$ and \[\mathcal{K}^{-1}(P^\bullet(S_y\langle 1\rangle)[1])\cong \mathcal{K}^{-1}(S_y\langle 1\rangle[1])\cong \mathcal{K}^{-1}(\mathcal{K}(I_y)\langle 1\rangle[1])\cong \mathcal{K}^{-1}\mathcal{K}(I_y\langle -1\rangle)\cong I_y\langle -1\rangle \]
where $I_x$ and $I_y$ are the injective envelopes of the simple $\Gamma^!$-modules $T_x$ and $T_y$ respectively.

We claim that the map $I_x\xto{f} I_y\langle -1\rangle$ is surjective. From this follows that $C$ is quasi-isomorphic to $\Ker f$ and then $\Delta_x\in \mathsf{gr}\Lambda^\uparrow\cap \mathcal{K}(\mathsf{gr}(\Lambda^!)^\downarrow)$ is a linear module.

The map $I_x\xto{f} I_y\langle -1\rangle$ is surjective if and only if the corresponding dual map between left $\Gamma^!$-modules \[\Gamma^!e_y\langle -1\rangle \to \Gamma^!e_x\] is injective and this map is given by right multiplication by the arrow $x\xto{\alpha_0}y$. If we put $B=(\Gamma^!)^{op}$, we can equivalently show that the map given by left multiplication by $\alpha_0$ between right projective $B$-modules is injective:
\[ e_yB\xto{\alpha_0\cdot}e_xB \]

To prove our claim we need the following really useful construction that we recall from the proof of Theorem 3.5 in \cite{twistCY}. Note that we will modify slightly the ideals of relations to adapt to our quasi-hereditary setting.

First of all note that, since $B=(\Gamma^!)^{op}$, with a little abuse of notation we can describe the quiver $Q$ of $B$ using the same notation as in Definition \ref{defin zigzag}:
\[Q_0= \{x=(x_0,x_1,\cdots,x_{n}) \in\Z_{\ge0}^{n+1} \mid \sum_{i=0}^{n+1} x_i=s\}\]
and 
\[Q_1 = \{x \xto{\alpha_i} x+f_i \mid i \in \{0, \ldots, n\}, x,x+f_i \in Q_0\}\]
where $f_i=(0,\ldots,\overset{i-1}{-1},\overset{i}{1},\ldots,0)$ and $f_{0}=(1,\ldots,-1)$. Remember moreover that we called $K$ the subset of $Q_0$ consisting of vertices $x$ such that $x_0=0$. Then $B=kQ/I$ where $I$ is the ideal generated by the elements \[e_x\alpha_i\alpha_j=\begin{cases} e_x\alpha_j\alpha_i \quad\text{ if }x+f_j\in Q_0 \\ 0\quad \text{ if }x+f_j\notin Q_0\text{ and }(i,j)\neq(0,1) \\ e_x\alpha_i\alpha_j \text{ if }x+f_j\notin Q_0\text{ and }(i,j)=(0,1) \end{cases}\]
for any $x\in Q_0$ such that $x+f_i,x+f_i+f_j\in Q_0$.
The elements $e_z\alpha_0\alpha_1$, for $z\in K$ are non-zero in $B$ (in contrast with the higher preprojective algebras described in \cite{twistCY}) since in $\Gamma$ we have $e_z\alpha_0\alpha_1=0$ and $B=(\Gamma^!)^{op}$.

Define the quiver $\widehat{Q}$ by
\[\widehat{Q}_0= \{x=(x_0,\cdots,x_{n}) \in\Z^{n+1} \mid \sum_{i=0}^{n+1} x_i=s, x_0\ge 0\}\]
\[\widehat{Q}_1 = \{x \xto{\alpha_i} x+f_i \mid i \in \{0, \ldots, n\}, x+f_i\in \widehat{Q}_0\}\]
and let $\widehat{I}$ be the ideal of $k\widehat{Q}$ defined by all the possible commutativity relations $\alpha_i\alpha_j=\alpha_j\alpha_i$. If we set $\widehat{B}=k\widehat{Q}/\widehat{I}$ then we have a surjective morphism of $k$-algebras $\pi:\widehat{B}\to B $ with kernel
\[R=\sum_{z\notin Q_0}\widehat{B}e_z\widehat{B} \]
and the residue classes of paths that are not in $R$ are mapped bijectively to residue classes of paths in $Q$.
For two paths in $\widehat{Q}$ $p,p'$ from $x$ to $y$, we will write $p\equiv p'$ if $p-p'\in \widehat{I}$. 

We define a $\Z^{n+1}$-grading $g$ on $\widehat{B}$ by $(g(\alpha_{i}))_j=\delta_{ij}$. This is a well-defined algebra grading on $\widehat{B}$ since $\widehat{I}$ is generated by homogeneous relations. Let $p$ be a path from $x$ to $y$, then $y-x=\sum_id_if_{i}$, where $d=(d_0,\cdots,d_{n})=g(p)$ is the degree of $p$. In fact we can always write $p\equiv p_{x,d,y}$ where
\[p_{x,d,y}=x\xto{\alpha_0}\cdots\xto{\alpha_0}x+d_0f_0\xto{\alpha_1}\cdots \xto{\alpha_{n-1}}y-d_{n}f_{n}\xto{\alpha_{n}}\cdots \xto{\alpha_{n}}y.\]
hence, in $\widehat{B}$, $p+\widehat{I}$ is determined by its degree $d$ and either $x$ or $y$. Moreover, for each path $p'$ from $x$ to $z$, we have $p\equiv p'q$ if and only if $g(p')_i\leq d_i$ for all $0\leq i\leq n$ (take for example $q=p_{z,d-g(p'),y}$). Hence the residue class $p+\widehat{I}$ is in the ideal $R$ if and only if $p\equiv p'q$ where $p'$ is a path from $x$ to $z$ and $z\notin Q_0$ and this is equivalent to say that there exist an index $j\neq 1$ such that $x_j<d_{j+1}$ where we work modulo $n+1$ on the indices (equivalently $p+\widehat{I}\notin R$ if and only if $x_j\ge d_{j+1}$ for all $j\neq 1$).

\begin{lemma}
Let $P_x=e_xB$ and $\pi\in\soc( P_x)$. Then $\pi$ corresponds to a maximal path starting at $x$ and ending at a vertex in $K$.
\end{lemma}
\begin{proof} Note that since $B$ is finite dimensional over $k$, for any $x\in Q_0$ there are (a finite number of) maximal paths in $Q$ starting at $x$, up to relations. We will show that any non-zero path $p$ from $x$ to $z$ with $z\notin K$ can be prolonged to a path ending in $z'\in K$. Let $x=(x_0,\cdots,x_{n})\in I$, $z=(z_0,\cdots,z_{n})\in J$ and $p$ a path in $\widehat{Q}$ from $x$ to $z$ such that $p+\widehat{I}\notin R$. Then $z'=z+z_0f_1\in K$ so, if $q=e_z\alpha_1^{z_0}$ is the path from $z$ to $z'$ given by the arrows $\alpha_1$ and $g(p)=d$, we have $d'=g(pq)=d+(0,z_0,\cdots,0)$. Since $p+\widehat{I}\notin R$ we have $x_j\ge d_{j+1}$ for every $j\neq 1$ and this implies $x_j\ge d_{j+1}'$. So $pq+\widehat{I}\notin R$ is a non-zero path from $x$ to $z'\in K$. \end{proof}

\begin{lemma}
$\soc(P_x)$ is simple for every $x\in I$.
\end{lemma}
\begin{proof} Let $p$ and $p'$ be two paths in $\widehat{Q}$ from $x$ to $z$ and $z'$ respectively, such that $p+\widehat{I},p'+\widehat{I}\notin R$. By the previous lemma we can suppose $z,z'\in K$. Let $g(p)=c$ and $g(p')=d$ so that we can write $z=x+\sum_ic_if_i$ and $z'=x+\sum_id_if_i$. Since the full subquiver of $Q$ that has $K$ as set of vertices is directed, there exists a vertex $z''\in K$ and two paths $q,q'$ in $K$ (hence not involving arrows $\alpha_0$ and $\alpha_1$) from $z$ and $z'$ respectively to $z''$. Since $q$ and $q'$ are paths in $K$ we have that $pq+\widehat{I},p'q'+\widehat{I}\notin R$ because $z_i\geq g(q)_{i+1}$ and $z'_i\geq g(q')_{i+1}$ for any $i\neq 0$. Hence they must coincide (up to equivalence) since $z''-x=\sum_ig(pq)_if_i=\sum_ig(p'q')_if_i$. Hence $dim_k\soc(P_x)=1$ and $\soc(P_x)$ is simple. \end{proof}

% there exists a path $q=p_{z,g(q),z'}\in\widehat{Q}$ from $z$ to $z'$ not involving any arrow $\alpha_0$ or $\alpha_1$. But then we have $z'-z=\sum_ig(q)_if_i=\sum_i(d'_i-d_i)f_i$ that implies $d_i'-d_i=g(q)_i\ge 0$ and this means exactly that $p'=pq$.

\begin{lemma}
Let $P_{x'}\xto{\alpha_0\cdot}P_x$ be the irreducible morphism between indecomposable projective $B$-modules given by left multiplication by $x\xto{\alpha_0}x'$. Then $\alpha_0\cdot(\soc(P_{x'}))\neq 0$.
\end{lemma}
\begin{proof} Let $p'$ be a path in $\widehat{Q}$ from $x'$ to $z$ such that $\pi(p'+\widehat{I})=p'+I$ generates $\soc(P_{x'})$. We want to show that $\alpha_0p'+\widehat{I}\notin R$. If $g(p')=d'$, then $g(\alpha_0p')=d'+(1,0,\cdots,0)=(d'_0+1,d_1',\cdots,d'_{n})$. Since $x'=x+f_0$, we have $x_i=x'_i$ for every $i\neq 0,n$; moreover $x'_{n}=x_{n}-1$ therefore $x_{n}=x'_{n}+1\ge d_{0}'+1$ and we conclude that $\alpha_0p'+\widehat{I}\notin R$.\end{proof}

As a consequence of the last lemma we have that any irreducible morphism between indecomposable projective $B$-modules is a monomorphism and the proof of the proposition is complete. \qed

\section{\texorpdfstring{$\Delta$}{Delta}-Koszulity}
In this section we want to investigate another kind of Koszul property. We start recalling the basic definitions and results by \cite{Madsen1}, \cite{Madsen2}.

\subsection{General results}

In what follows we will consider a grading that is different from the radical grading. To avoid confusion in the notation, given a graded algebra $\Lambda$, we will denote its graded subspaces by $\Lambda_{[i]}$ whenever the grading is not the radical grading. Later on (Section \ref{sec delta kosz}) this grading will coincide with the $\langle\cdot\rangle^\flat$-grading.

\begin{defin}\cite{Madsen1} \label{deltakoszul}
Let $\Lambda$ be a graded algebra such that $\gldim\Lambda_{[0]}<\infty$ and let $T$ be a graded $\Lambda$-module concentrated in degree zero. Then we say that $\Lambda$ is \emph{Koszul with respect to T} or \emph{T-Koszul} if:
\begin{enumerate}
\item $T$ is a tilting $\Lambda_{[0]}$-module.
\item $T$ is graded self-orthogonal as a $\Lambda$-module, that is 
\[ \ext_{\mathsf{gr}\Lambda}^i(T,T\langle j\rangle)=0,  \mbox{ whenever } i\neq j.\]
\end{enumerate}
\end{defin}

We recall the following results about graded self-orthogonal modules:
\begin{lemma} \cite[Proposition 3.1.2, Corollary 3.1.3] {Madsen1}\label{graded self orth}
Let $\Lambda$ be a graded $k$-algebra with degree zero part not necessarily semisimple and $T$ a graded self-orthogonal module. Then
\[ \ext^i_\Lambda(T,T)\cong \ext^i_{\mathsf{gr}\Lambda}(T,T\langle i \rangle)\]
for each $i\ge 0$. Moreover there is an isomorphism of graded algebras 
\[\bigoplus_{i\ge 0}\ext^i_{\Lambda}(T,T)\cong\bigoplus_{i\ge 0}\ext^i_{\mathsf{gr}\Lambda}(T,T\langle i\rangle). \] 
\end{lemma}

An analogous of Koszul duality holds for $T$-Koszul algebras:

\begin{thm}\cite[Theorem 4.2.1]{Madsen1}\label{delta dual}
Let $\Lambda$ be a graded $k$-algebra such that $\gldim\Lambda_{[0]}<\infty$ and suppose that $\Lambda$ is $T$-Koszul for a module $T$. Let $\Lambda^\dagger=\ext^*_\Lambda(T,T)$ endowed with the $\ext$-grading, then:
\begin{enumerate}
\item $\gldim \Lambda^\dagger_{[0]}<\infty$ and $\Lambda^\dagger$ is Koszul with respect to $D T_{\Lambda^\dagger}$.
\item There is an isomorphism of graded algebras $\Lambda\simeq\ext^*_{\Lambda^\dagger}(DT,DT)$
\end{enumerate}
\end{thm}
If this is the case we say that the pair $(\Lambda^\dagger,DT)$ is the \emph{Koszul dual} of $(\Lambda,T)$.

When $\Lambda$ is $T$-Koszul there exists a complex of bigraded $\Lambda$-$\Lambda^\dagger$-modules $X$ that defines two functors
\[ F_T=X\otimes_{\mathsf{gr}\Lambda^\dagger}^\mathbb{L}-:\mathcal{D}(\mathsf{gr}\Lambda^\dagger)\rightleftharpoons \mathcal{D}(\mathsf{gr}\Lambda):G_T=\mathbb{R}\Hom_{\mathsf{gr}\Lambda}(X,-) \]
such that $(F_T,G_T)$ is an adjoint pair (as explained in \cite{Madsen1}, Section 3 and \cite{kellerDG}).

In general the two functors above are not quasi-inverses one to the other but they induce an equivalence on certain subcategories. Let $\mathcal{F}_{\mathsf{gr}\Lambda}$ be the full subcategory of $\mathsf{gr}\Lambda$ of modules $M$ having a finite filtration $0=M_0\subseteq M_1\subseteq\cdots\subseteq M_t=M$ with factors that are graded shifts of direct summands of $T$. Let $\mathcal{L}^b(\Lambda^\dagger)$ be the category of bounded linear complexes of graded projective $\Lambda^\dagger$-modules.
\begin{thm}\cite[Theorems 4.3.2, 4.3.4]{Madsen1}\label{equivalences}
The functor $G_T:\mathcal{D}(\mathsf{gr}\Lambda)\to\mathcal{D}(\mathsf{gr}\Lambda^\dagger)$ restricts to an equivalence $G_T:\mathcal{F}_{\mathsf{gr}\Lambda}\to\mathcal{L}^b(\Lambda^\dagger)$. If moreover $\Lambda$ is Artinian, $\Lambda^\dagger$ is Noetherian and $\gldim \Lambda^\dagger<\infty$, then there is an equivalence of triangulated categories $G_T^b:\mathcal{D}^b(\mathsf{gr}\Lambda)\to\mathcal{D}^b(\mathsf{gr}\Lambda^\dagger)$ between the bounded derived categories.
\end{thm}

The following Proposition  gives some useful properties of the adjoint pair $(F_T,G_T)$:
\begin{prop}\cite[Proposition 3.2.1]{Madsen1} \label{delta koszul propert}
Let $T$ be a graded self-orthogonal $\Lambda$-module, $M$ a finitely cogenerated $\Lambda$-module and $N$ an object in $\mathcal{D}(\mathsf{gr}\Lambda)$. Then, for every $i,j\in\mathbb{Z}$, we have:
\begin{enumerate}[(a)]
\item $G_T(T)\cong \Lambda^\dagger$.
\item If $\phi:F_TG_T\to id_{\mathcal{D}(\mathsf{gr}\Lambda)}$ is the counit of the adjunction, then $\phi_T:F_TG_T(T)\to T$ is an isomorphism.
\item There is a functorial isomorphism $G_T(N\langle j\rangle)\cong G_T(N)\langle-j\rangle[-j]$.
\item There is a functorial isomorphism $F_TG_T(N\langle j\rangle)\cong F_TG_T(N)\langle j\rangle$.
\item $(H^iG_T(M))_j\cong \ext_{\mathsf{gr}\Lambda}^{i+j}(T,M\langle j\rangle)$.
\item $G_T(D\Lambda)\cong DT_{\Lambda^\dagger}$.
\end{enumerate}
\end{prop}

We are particularly interested in the case when a quasi-hereditary algebra is Koszul with respect to the standard module $\Delta$.
Example 2.4 and 4.7 in \cite{Madsen2} show that if $Z$ is the Brauer algebra associated to the Brauer line, then its qh-cover $\Gamma$ is standard Koszul with grading given by path-length but it is also possible to define another grading in order to make it $\Delta$-Koszul.

\begin{exm} \cite{Madsen1} \cite{Madsen2}
In the case of the Brauer line algebra $Z^{(1,s)}$, the quiver of its qh-cover $\Gamma$ is:
\[ \xymatrix{1 \ar@/^/[r]^\alpha & 2 \ar@/^/[l]^\beta\ar@/^/[r]^\alpha & \cdots \ar@/^/[l]^\beta \ar@/^/[r]^\alpha & s \ar@/^/[r]^\alpha\ar@/^/[l]^\beta & s+1 \ar@/^/[l]^\beta} \] 
bound by the ideal of relations $I=\left( \alpha\beta-\beta\alpha, \alpha^2, \beta^2, e_{s+1}\beta\alpha \right)$.
We can see that $\Gamma$ is Koszul and standard Koszul or Koszul with respect to $\Delta$ depending on the grading that we put on the algebra:
\begin{enumerate}
\item If all the arrows are given degree $1$ then $\Gamma$ is Koszul in the classical sense and standard Koszul.
\item If we put $deg_\Delta\alpha=1$ and $deg_\Delta\beta=0$, this define an algebra grading on $\Gamma$ and the conditions of Definition \ref{deltakoszul} are satisfied with $T=\Gamma_{[0]}=\Delta$. Hence $\Gamma$ is Koszul with respect to $\Delta$ and the Koszul dual algebra ${\Gamma^\dagger}=\ext^*_\Gamma(\Delta,\Delta)$ is the path algebra of the following quiver:
\[ \xymatrix{1 & 2 \ar@/^/[l]^\beta\ar@/_/[l]_{\alpha^*} & \ldots \ar@/^/[l]^\beta \ar@/_/[l]_{\alpha^*} & s \ar@/_/[l]_{\alpha^*}\ar@/^/[l]^\beta & s+1 \ar@/_/[l]_{\alpha^*}\ar@/^/[l]^\beta} \] 

with relations $I'=\left( \beta^2, \alpha^*\beta-\beta\alpha^* \right)$. Moreover it is shown in \cite{Madsen2} that ${\Gamma^\dagger}$ is Koszul in the classical sense.

\end{enumerate}
\end{exm}

\subsection{\texorpdfstring{$\Delta$}{Delta}-Koszulity of qh-covers}

Motivated by the Brauer line case, we want to show that similar results are true for qh-covers of higher zigzag-algebras (see Theorem 4.1 and 4.4 of \cite{Madsen2}). We already know that if $Z$ is a higher zigzag-algebra, then its quasi-hereditary cover $\Gamma$ is Koszul and standard Koszul. Now we want to prove that $\Gamma$ is Koszul with respect to $\Delta$.

First let us recall some homological properties of quasi-hereditary algebras, so let $\Lambda$ be quasi-hereditary with set of weights $I$. For two vertices $i,j\in I$ such that $i\leq j$ the \emph{distance} between them is defined as
\[ d(i,j)=max\{n\in\mathbb{N} : \exists i=i_0<\ldots<i_n=j \}. \]
If $i$ and $j$ are not comparable we set $d(i,j)=\infty$.
Note that, in the case of $\Gamma$, the distance between two vertices $x,y$ is precisely the length of a minimal path $\pi$ from $x$ to $y$ that does not involve arrows of the form $\alpha_0$. Moreover the length of such a path in $Q^{(n,s)}$ is unique and this implies that the distance is additive: if $d(x,y)=h$ and $d(y,z)=k$ then $d(x,z)=h+k$.

\begin{lemma}\cite[Lemma 3]{Far} \label{homprop}
If $\Lambda$ is a quasi-hereditary algebra, then the following hold:
\begin{enumerate}
\item If $i>j$ then $\Hom_\Lambda(\Delta_i,\Delta_j)=0$.
\item If $l>0$ and $i\nless j$ then $\ext_\Lambda^l(\Delta_i,S_j)\simeq\ext_\Lambda^l(\Delta_i,\Delta_j)=0$.
\item If $i\leq j$ and $l>d(i,j)$ then $\ext_\Lambda^l(\Delta_i,S_j)\simeq\ext_\Lambda^l(\Delta_i,\Delta_j)=0$. 
\end{enumerate}
\end{lemma}

Now we define a new grading on $\Gamma$, that we will denote by $deg_\Delta(-)$, by setting
\[ deg_\Delta(e_x)=0 \quad \forall x\in I, \quad deg_\Delta(\alpha_k)=\begin{cases} 1 \mbox{ if } k\neq 0 \\ 0 \mbox{ if } k=0 \end{cases}
\]
The fact that this is a well defined algebra grading is assured by the fact that every non-monomial relation is of the form $\alpha_i\alpha_j=\alpha_j\alpha_i$.
We call this grading the $\Delta$-grading (according to \cite{Madsen2}) and we denote by $\Gamma_{[i]}$ the $\Delta$-degree $i$ part of $\Gamma$. Then we have the following:

\begin{prop}\label{grading prop}
Let $\Gamma$ be our qh-cover of the higher zigzag-algebra $Z$.
\begin{enumerate}
\item Consider $\Gamma$ with the ordinary grading. If $\ext_{\mathsf{gr}\Gamma}^u(\Delta_y,S_x\langle v\rangle)\neq 0$ then $u=v=d(x,y)$.
\item According to the $\Delta$-grading, $\Gamma_{[0]}\cong \Delta$ as graded $\Gamma$-modules.
\item If $\Gamma$ is given the $\Delta$-grading, then minimal resolutions of standard modules are linear with respect to the $\Delta$-grading.
\end{enumerate}
\end{prop}

\begin{proof} 
\begin{enumerate}
\item Suppose that $\ext_{\mathsf{gr}\Gamma}^u(\Delta_y,S_x\langle v\rangle)\neq 0$. Since $\Gamma$ is standard Koszul we have $u=v$. Consider a linear projective resolution of $\Delta_y$:
\[\ldots\to P^u\to\ldots \to P^1\to P^0=P_y\to\Delta_y\to 0 \]
where the indecomposable projective module $P_x$ appears as a direct summand in $P^u$. By part $(2)$ of \ref{homprop} we have that $y<x$ and so, by the definition of the partial order on the set of weights of $\Gamma$, there exists a path $\pi$ from $y$ to $x$ that involves only arrows $\alpha_k$ for $k\neq 0$. The length of this path $\pi$ is equal to $d(x,y)$ and so we have $d(x,y)\leq u$.
On the other hand by part $(3)$ of \ref{homprop} we deduce that $u\leq d(x,y)$ and this proves the claim.

\item This follows from the definition of $\Delta$-grading.

\item By what has been proved in part $(1)$, if $P^u\to P^v$ is a map in a linear projective resolution of a standard module $\Delta_x$, then the image of generators of $P^u$ in $P^v$ are linear combinations of elements of the form $e_a\alpha_ke_b$ with $k\neq 0$, since $d(a,b)=1$. Then the resolution is also linear with respect to the $\Delta$-grading. 
\end{enumerate} \end{proof}

From the above results we have the following theorem:

\begin{thm}\label{delta koszul}
Consider $\Gamma$ as a graded algebra according to the $\Delta$-grading. Then $\Gamma$ is Koszul with respect to $\Delta$.
\end{thm}

\begin{proof}
The algebra $\Gamma_{[0]}$ can be decomposed in subalgebras each of which is isomorphic to a type $A$ algebra with underlying quiver:
\[ x_1\xto{\alpha_0}x_2\xto{\alpha_0}\cdots \xto{\alpha_0} x_k \]
bound by relations $\alpha_0\alpha_0=0$, for $1\leq k\leq s+1$. These algebras have all finite global dimension, hence $\Gamma_{[0]}$ has finite global dimension as well. $\Delta$ is a tilting $\Gamma_{[0]}$-module by part $(2)$ of Proposition \ref{grading prop}. Now let $P^i$ be a projective module in a minimal graded projective resolution of $\Delta$; then $P^i$ is generated in degree $i$ and since $\Delta\langle j\rangle$ is concentrated in degree $j$ we have that $\Hom_{\mathsf{gr}\Gamma}(P^i,\Delta\langle j\rangle)=0$ if $i\neq j$. Hence  $\ext_{\mathsf{gr}\Gamma}^i(\Delta, \Delta\langle j\rangle)= 0$ whenever $i\neq j$. \end{proof}

The same argument that has been used in the proof of Theorem \ref{delta koszul} to show that $\Delta$ is graded self-orthogonal can be stated in a more general way.

\begin{lemma}\label{koszulmodules}
Let $\Lambda=\bigoplus_{i\geq 0}\Lambda_{[i]}$ be a graded algebra (with $\Lambda_{[0]}$ not necessarily semisimple) and $T$ a finitely generated $\Lambda$-module concentrated in degree zero. 
\begin{enumerate}
\item If $T$ is linear and then it is graded self-orthogonal. 
\item If $\ext^i_{\mathsf{gr}\Lambda}(T,\Lambda_{[0]}\langle j\rangle)=0$ unless $i=j$, then $T$ is linear.
\end{enumerate}
\end{lemma} 

\begin{proof} The first statement is clear from the proof of Theorem \ref{delta koszul}.

The proof of the second statement can be found in \cite{BSGkoszul}, Proposition 2.14.2, in the case when $\Lambda_{[0]}$ is semisimple and the part of the proof we are interested in is still true without any assumption on $\Lambda_{[0]}$. We include the proof here for the convenience of the reader.

$T$ is concentrated in degree zero over $\Lambda$, so that its projective cover consists of a projective module $P^0$ generated in degree zero. We want to find a linear graded projective resolution for $T$ by induction, so let us assume that we have a projective resolution 
\[P^i\to P^{i-1}\to\cdots \to P^0\to T \to 0\]
such that $P^i$ is generated in degree $i$ over $\Lambda$ and the differential is injective on the degree $i$ part of $P^i$, $P^i_{[i]}$. Then if we put $K=\Ker(P^i\to P^{i-1})$, we have that $K_{[j]}=0$ for $j<i+1$. If $N$ is any $\Lambda$-module that is concentrated in one single degree then $\ext^{i+1}_{\mathsf{gr}\Lambda}(T,N)=\Hom_{\mathsf{gr}\Lambda}(K,N)$. But then our assumption means that $\Hom_{\mathsf{gr}\Lambda}(K,\Lambda_{[0]}\langle j\rangle)=0$ unless $i+1=j$, that is, $K$ is generated in degree $i+1$ over $\Lambda$. Then we can find a projective cover $P^{i+1}$ of $K$ that is generated in degree $i+1$ and we can conclude by induction.  \end{proof}

It is important to underline that in general the two conditions in part (1) of Lemma \ref{koszulmodules} are not equivalent as the following example shows.
\begin{exm}
Let $\Lambda$ be the path algebra of the following quiver:
\[ \xymatrix{1  & 2 \ar@/^/[l]^a \ar@/_/[l]_c & 3 \ar@/_/[l]_d \ar@/^/[l]^b } \] 
with relations $ba=0$, $da=bc$. Define a grading $|\cdot|$ on $\Lambda$ by setting $|a|=|b|=0$ and $|d|=|c|=1$ and let $\Lambda_{[0]}$ be the subalgebra of $\Lambda$ concentrated in degree zero. Then $\Lambda$ is Koszul with respect to $D\Lambda_{[0]}$ but the simple module $S_3$ is a direct summand of $D\Lambda_{[0]}$ and its (graded) projective resolution is:
\[ 0 \to P_1\oplus P_1\langle 1\rangle \to P_2\oplus P_2\langle 1\rangle \to P_3 \to S_3\to 0 \] 
hence it is not linear.
\end{exm}

%\begin{lemma}\label{self-ortog}
%Let $\Lambda$ be a positively graded algebra and $T$ a finitely generated (graded) $\Lambda$-module. Then $T$ is graded self-orthogonal if and only if $DT$ is.
%\end{lemma}
%\begin{proof}
%For any $i,j\ge 0$ we have:
%\[\begin{split} \ext^j_{\mathsf{gr}\Lambda}(DT,DT\langle i\rangle) & \cong \Hom_{\mathcal{D}^b{\mathsf{gr}\Lambda}}(DT,DT\langle i\rangle [j]) \\ & \cong \Hom_{\mathcal{D}^b{\mathsf{gr}\Lambda}}(DT,\Hom_k(T\langle -i\rangle, k) [j]) \\ & \cong \Hom_k(T\langle -i \rangle \otimes DT,k [j]) \\ & \cong \Hom_{\mathcal{D}^b{\mathsf{gr}\Lambda}}(T\langle -i\rangle ,\Hom_k(DT,k) [j]) \\ & \cong \Hom_{\mathcal{D}^b{\mathsf{gr}\Lambda}}(T\langle -i\rangle ,T [j]) \cong \ext^j_{\mathsf{gr}\Lambda}(T,T\langle i\rangle)\end{split} \]
%where we have used the adjointness of the two pairs of functors $(T\otimes_k-,\Hom_k(T,-))$ and $(-\otimes_k DT,\Hom_k(DT,-))$. \end{proof}\\

\section{\texorpdfstring{$\Delta$}{Delta}-Koszul duality} \label{sec delta kosz}

In this last section we want to study the $\Delta$-Koszul dual ${\Gamma^\dagger}$ when $\Gamma$ is the qh-cover that we defined for a higher zigzag-algebra. First we want to show that ${\Gamma^\dagger}$ is a bigraded algebra and that it is Koszul when endowed with the total grading $|\cdot|^{tot}$. This is true when we consider the qh-cover of the Brauer line by \cite[Theorem 4.4]{Madsen2} and the proof is based on the existence of a particular \textit{height function} on the set of vertices of the quiver of $\Lambda$ (see \cite{Madsen2}). It is then reasonable to try to generalize this result for higher zigzag-algebras (of type $A$). To conclude we compute the quiver of the $\Delta$-dual algebra and, using the fact that Koszulity implies quadraticity, we determine its ideal of relations.

\subsection{Bigraded \texorpdfstring{$\Delta$}{Delta}-Koszul algebras}
Suppose we can define two gradings $|\cdot|^\flat$ and $|\cdot|^\sharp$ on $\Lambda$, with shifts $\langle\cdot\rangle^\flat$ and $\langle\cdot\rangle^\sharp$, and corresponding categories of graded modules $\mathsf{gr}^\flat\Lambda$ and $\mathsf{gr}^\sharp\Lambda$ respectively. Let $|\cdot|^{tot}$ be the total grading on $\Lambda$ obtained by adding the $|\cdot|^\flat$-degree and the $|\cdot|^\sharp$-degree. For $i\geq 0$, denote by $\Lambda_{[i]}$ the degree-$i$ subspace of $\Lambda$ with respect to $|\cdot|^\flat$, and by $\Lambda_i$ the degree-$i$ subspace of $\Lambda$ with respect to $|\cdot|^{tot}$; then $\Lambda_0\subseteq \Lambda_{[0]}$. Suppose moreover that $(\Lambda,|\cdot|^\flat)$ is $\Lambda_{[0]}$-Koszul and let $({\Lambda^\dagger},D\Lambda_{[0]})$ be the Koszul dual of $(\Lambda,\Lambda_{[0]})$.

The grading $|\cdot|^\sharp$ on $\Lambda$ induces a grading on ${\Lambda^\dagger}$ in the following way.
Since $\Lambda_{[0]}$ is concentrated in $|\cdot|^\flat$-degree zero, the $|\cdot|^\sharp$-degree on $\Lambda_{[0]}$ coincide with $|\cdot|^{tot}$, so $\Lambda_{[0]}$ inherits a graded structure from $|\cdot|^\sharp$ by defining the graded parts $(\Lambda_{[0]})_n=\Lambda_n\cap \Lambda_{[0]}$. Put $V_{n,j}=\ext^n_{\mathsf{gr}^\sharp\Lambda}(\Lambda_{[0]},\Lambda_{[0]}\langle j\rangle^\sharp)$; the Yoneda extension groups of $\Lambda_{[0]}$ are graded $k$-vector spaces:
\[ \ext^n_\Lambda(\Lambda_{[0]},\Lambda_{[0]})=\bigoplus_{j\ge 0}\ext^n_{\mathsf{gr}^\sharp\Lambda}(\Lambda_{[0]},\Lambda_{[0]}\langle j\rangle^\sharp)=\bigoplus_{j\ge0}V_{n,j}\]
%=\bigoplus_{j\ge 0}\ext^n_\Lambda(A,A)_j=
Setting $V_{\bullet,j}=\bigoplus_{n\geq 0}V_{n,j}$ gives a grading on ${\Lambda^\dagger}=\bigoplus_{j\geq 0}V_{\bullet,j}$ that we will denote again by $|\cdot |^\sharp$.
The algebra ${\Lambda^\dagger}=\ext^\ast_\Lambda(\Lambda_{[0]},\Lambda_{[0]})$ is also a graded algebra with respect to the $\ext$-grading since, for any $n,m\ge0$ we have
%\[ \begin{split}
%\ext^n_\Lambda(A,A)\ext^m_\Lambda(A,A) & =\left(\oplus_{i\ge %0}\ext^n_\Lambda(A,A)_i\right)\left(\oplus_{j\ge 0}\ext^m_\Lambda(A,A)_j\right) \\ & \cong \oplus_{k\ge 0}\oplus_{i+j=k}\ext^n_\Lambda(A,A)_i\ext^m_\Lambda(A,A)_j \\ & \subseteq \oplus_{k\ge 0}\ext^{n+m}_\Lambda(A,A)_k = \ext^{n+m}_\Lambda(A,A) \end{split}.\] 
\[ \ext^n_\Lambda(\Lambda_{[0]},\Lambda_{[0]})\ext^m_\Lambda(\Lambda_{[0]},\Lambda_{[0]})\subseteq \ext^{n+m}_\Lambda(\Lambda_{[0]},\Lambda_{[0]}) \]
Note that, since $\Lambda_{[0]}$ is graded self-orthogonal with respect to $|\cdot|^\flat$, the $\ext$-grading on ${\Lambda^\dagger}$ is precisely the one induced by $|\cdot |^\flat$; hence we will denote the $\ext$-grading on $\Lambda^\dagger$ again by $|\cdot |^\flat$. The decomposition of $\Lambda^\dagger$ in bigraded subspaces is $\Lambda^\dagger=\bigoplus_{n,j\ge0}V_{n,j}=\bigoplus_{n\ge0}
\left(\bigoplus_{j\ge0}V_{n,j}\right)$. We can define $V_n=\bigoplus_{i+j=n}V_{i,j}$ so that
\[\begin{split}
V_{0,0} & =\Hom_{\mathsf{gr}^\sharp\Lambda}(\Lambda_{[0]},\Lambda_{[0]})(\cong\Lambda_0), \\
V_{0,1} & =\Hom_{\mathsf{gr}^\sharp\Lambda}(\Lambda_{[0]},\Lambda_{[0]}\langle 1\rangle^\sharp), \\
V_{1,0} & =\ext^1_{\mathsf{gr}^\sharp\Lambda}(\Lambda_{[0]},\Lambda_{[0]}), \\
\ldots
\end{split}\]
Then we can write ${\Lambda^\dagger}=\bigoplus_{n\ge0}V_n$ and this defines a new graded structure on ${\Lambda^\dagger}$ as a $k$-vector space. From the above we have that
\[\begin{split}
V_nV_m & =
\left(\bigoplus_{i+j=n}\ext^i_{\mathsf{gr}^\sharp\Lambda}(\Lambda_{[0]},\Lambda_{[0]}\langle j \rangle^\sharp)\right) 
\left(\bigoplus_{{\vphantom{j}}h+l=m}\ext^h_{\mathsf{gr}^\sharp\Lambda}(\Lambda_{[0]},\Lambda_{[0]}\langle l \rangle^\sharp)\right) 
\\ & \subseteq \bigoplus_{i+j+h+l=n+m}\ext^{i+h}_{\mathsf{gr}^\sharp\Lambda}(\Lambda_{[0]},\Lambda_{[0]}\langle j+l \rangle^\sharp) \end{split} \]
hence this gives us a graded structure on ${\Lambda^\dagger}$ as a $k$-algebra. Finally we will denote this grading on ${\Lambda^\dagger}$ by $|\cdot|^{tot}$ and the category of (finitely generated) graded modules by $\mathsf{tgr}\Lambda$.

The first result of the following is essentially Proposition 4.2 of \cite{Madsen2} when $\Lambda$ is quasi-hereditary and $\Delta$-Koszul, with $\Delta$-grading given by $|\cdot|^\flat$ so that $\Lambda_{[0]}=\Delta$. Recall that $G_\Delta=\Hom_{\D(\mathsf{gr}^\flat\Lambda)}(\Delta,-)$.

\begin{prop}\cite{Madsen2} \label{simple costand}
Let $\Lambda$ be a bigraded quasi-hereditary algebra, with gradings $|\cdot |^\flat$ and $|\cdot |^\sharp$ as before, that is also $\Delta$-Koszul with respect to $|\cdot |^\flat$. Then
\begin{enumerate}
\item $ G_\Delta(\nabla_x)\cong S_x$
\item $ S_x\langle j\rangle^\sharp \cong G_\Delta(\nabla_x\langle j\rangle^\sharp)$
\end{enumerate}
where $\nabla_x$ denotes the costandard $\Lambda$-module of weight $x$ and $S_x$ is the simple ${\Lambda^\dagger}$-module whose projective cover is $G_\Delta(\Delta_x)$.
\end{prop}
The original statement in \cite{Madsen2} is about standard Koszul algebras admitting a particular height function but the proof is still valid in the case of $\Delta$-Koszul algebras. We include the original argument here for the convenience of the reader.
\begin{proof}\begin{enumerate}
\item By Proposition \ref{delta koszul propert}(e), we have 
\[(H^kG_\Delta(\nabla_x))_j\cong\ext^{k+j}_{\mathsf{gr}^\flat\Lambda}(\Delta,\nabla_x\langle j\rangle^\flat)=0 \]
whenever $k\neq 0$ or $j\neq 0$. Then
\[ \begin{split}
G_\Delta(\nabla_x) & \cong (H^0(G_\Delta(\nabla_x))_0 \\
                   & \cong \Hom_{\mathsf{gr}^\flat\Lambda}(\Delta,\nabla_x)\\
                   & \cong \Hom_{\mathsf{gr}^\flat\Lambda}(\Delta_x,\nabla_x), \end{split} \]
that is a one-dimensional $k$-vector space. Moreover, if $y\neq x$,
\[\Hom_{\mathcal{D}(\mathsf{gr}^\flat{\Lambda^\dagger})}(G_\Delta(\Delta_y),G_\Delta(\nabla_x))\cong\Hom_{\mathcal{D}(\mathsf{gr}^\flat\Lambda)}(\Delta_y,\nabla_x)=0 \]
so we must have $G_\Delta(\nabla_x)\cong \modtop G_\Delta(\Delta_x)$.
\item We have 
\[ \begin{split} S_x\langle j\rangle^\sharp & \cong G_\Delta(\nabla_x)\langle j\rangle^\sharp \\ & \cong\Hom_{\mathcal{D}(\mathsf{gr}^\flat\Lambda)}(\Delta,\nabla_x)\langle j\rangle^\sharp \\  & \cong\bigoplus_{k\in\mathbb{Z}}\Hom_{\mathcal{D}(\mathsf{tgr}\Lambda)}(\Delta,\nabla_x\langle 0,k+j\rangle)\cong G_\Delta(\nabla_x\langle j\rangle^\sharp) \end{split} \]
\end{enumerate} \end{proof}

Let us describe the bigraded structure that we will consider on $\Gamma$ and on its $\Delta$-dual $\Gamma^\dagger$. Denote by $|\cdot|^\flat$ the $\Delta$-grading on $\Gamma$ and recall that, when defining the $\Delta$-grading, we denoted by $\Gamma_{[0]}$ the degree zero part of $\Gamma$ with respect to this grading. We can define another grading $|\cdot|^\sharp$ on $\Gamma$ such that the total grading correspond to the radical grading:
\[ |e_x|^\sharp=0 \quad \forall x\in I, \quad  |\alpha_k|^\sharp=\begin{cases} 0 \mbox{ if } k\neq 0 \\ 1 \mbox{ if } k=0. \end{cases}
\]
When considering the dual algebra $\Gamma^\dagger$, we will denote the $\ext$-grading by $|\cdot|^\flat$ (since it is induced by the $\Delta$-grading) and the grading induced by $|\cdot|^\sharp$ always by $|\cdot|^\sharp$. For every bigraded $\Gamma$-module (or $\Gamma^\dagger$ in the same way) $M$, we will denote by $M\langle i,j\rangle$ the bigraded module obtained by shifting $M$ of $i$ with respect to $|\cdot|^\flat$ and of $j$ with respect to $|\cdot|^\sharp$. Then we will denote by $|\cdot|^{tot}$ the total grading on $\Gamma$ (and on $\Gamma^\dagger$ similarly).

\subsection{\texorpdfstring{$\Delta$}{Delta}-Koszul dual of qh-covers}

Let ${\Gamma^\dagger}=\ext^*_\Gamma(\Delta,\Delta)$, so by Theorem \ref{delta dual} ${\Gamma^\dagger}$ is Koszul with respect to $D\Delta$ and $({\Gamma^\dagger},D\Delta)$ is the Koszul dual of $(\Gamma,\Delta)$. The  $|\cdot|^\flat$-degree zero part of $\Gamma^\dagger$ is $\End_\Gamma(\Delta)\simeq\End_{\Gamma_{[0]}}(\Delta)\simeq\Delta_{{\Gamma^\dagger}}$ considered as a right $\Gamma^\dagger$-module. We have the following corollary to Theorems \ref{delta koszul}, \ref{delta dual} and \ref{equivalences}:

\begin{cor}
There is an isomorphism 
\[\Gamma\cong\ext^*_{\Gamma^\dagger}(D\Delta,D\Delta)\]
as ungraded algebras. Moreover, if $\Gamma$ is given the $\Delta$-grading and ${\Gamma^\dagger}$ the $\ext$-grading, then there is an equivalence of triangulated categories $G_\Delta:\mathcal{D}^b(\mathsf{gr}^\flat\Gamma)\to \mathcal{D}^b(\mathsf{gr}^\flat{\Gamma^\dagger})$ 
which restricts to an equivalence  $G_\Delta:\mathcal{F}_{\mathsf{gr}^\flat\Gamma}(\Delta)\to \mathcal{L}^b({\Gamma^\dagger})$.
\end{cor}
\begin{proof}
Since $\Gamma$ with the $\Delta$-grading is Koszul with respect to $\Delta$ the isomorphism follows from Theorem \ref{delta dual}. By Theorem \ref{delta koszul} there is an equivalence $G_\Delta:\mathcal{D}^b(\mathsf{gr}^\flat\Gamma)\to \mathcal{D}^b(\mathsf{gr}^\flat{\Gamma^\dagger})$. Moreover since $\Gamma$ has finite global dimension, ${\Gamma^\dagger}$ is finite dimensional and it is directed since the extension algebra of standard modules is always directed \cite[Theorem 1.8(b)]{parshall1998}. Then ${\Gamma^\dagger}$, being directed, has finite global dimension and, since obviously $\Delta\in\mathcal{D}^b(\mathsf{gr}^\flat\Gamma)$, the category $\mathcal{F}_{\mathsf{gr}^\flat\Gamma}(\Delta)$ is a subcategory of $\mathcal{D}^b(\mathsf{gr}^\flat\Gamma)$. Then the claims follow from Theorem \ref{equivalences}. 
\end{proof}

The following lemma gives a useful description of the graded parts of ${\Gamma^\dagger}=\ext^*_\Gamma(\Delta,\Delta)$ when $\Gamma$ is the qh-cover of a higher zigzag algebra.

\begin{lemma}\label{gamma totgrad}
Let $x,y$ be two vertices in the quiver of $\Gamma$ and $d=d(x,y)$ their distance in the quiver. If $\ext_{\mathsf{gr}^\sharp\Gamma}^i(\Delta_x,\Delta_y\langle j \rangle^\sharp)\neq 0$ for some $i,j\ge 0$, then $i=d-nj$. As a consequence we have that
\[ |\ext_\Gamma^i(\Delta_x,\Delta_y\langle j \rangle^\sharp) |^{tot}=d-j(n-1). \]
\end{lemma}
\begin{proof}
We proceed by induction on $d=d(x,y)$. Recall that $d(x,y)$ is the length of a path from $x$ to $y$ not involving any arrow $\alpha_0$, if such a path exists, and it is $\infty$ otherwise. The distance between two vertices is infinite precisely when they are not comparable in the partial order on the set of vertices. But this can not happen under our assumptions since, by Lemma \ref{homprop}(2), $\ext_\Gamma^i(\Delta_x,\Delta_y)\neq 0$ implies $x< y$, for any $i>0$.

Note first that if $d=0$ then $x=y$ and $\ext^*_\Gamma(\Delta_x,\Delta_x)\cong\Hom_\Gamma(\Delta_x,\Delta_x)$ by quasi-heredity.

Suppose $d=1$: by Lemma \ref{homprop}, if $x>y$ then $\ext^i_\Gamma(\Delta_x,\Delta_y)=0$, so we can assume $x<y$. In this case $\ext^i_\Gamma(\Delta_x,\Delta_y)\neq 0$ only if $i\leq d=1$. We must have $i\neq 0$ since $d=1$ implies that there exists an arrow $x\xto{\alpha_k}y$ with $k\neq 0$. So, since $n>1$, there can not be an arrow $y\xto{\alpha_0}x$ and $\Hom_\Gamma(\Delta_x,\Delta_y)=0$. Therefore $j=0$ and $i=d=1$.

Assume now $d>1$ and, for any vertex $y'$ such that $d'=d(x,y')<d$,  if $\ext^i_\Gamma(\Delta_x,\Delta_{y'}\langle j \rangle^\sharp)\neq 0$ then   $i=d'-nj$.
Let $P^\bullet(x)$ be a projective resolution of $\Delta_x$, linear with respect to the $\Delta$-grading on $\Gamma$, and consider a morphism $f:P^i(x)\to\Delta_y$ that gives a non-zero homogeneous element in $\ext^i_\Gamma(\Delta_x,\Delta_y\langle j \rangle^\sharp)$. If $(S_y,S_{y'})$ are the composition factors of $\Delta_y$ then $d'=d(x,y')=d-n$.
We distinguish two cases:
\begin{itemize}
\item If $f$ is epi, then $P_y$ is a direct summand of $P^i(x)$ and so $i=d$ and $j=0$.
\item If $f$ is not epi then it factors through the morphism $\Delta_{y'}\to\Delta_y$ and we have the following commutative diagram:
\[\xymatrix{P^i(x) \ar[d]_g \ar[dr]^f & \\ \Delta_{y'} \ar[r] & \Delta_y} \] where $g$ is non-zero and epi, hence it belongs to $\ext^i_\Gamma(\Delta_x,\Delta_{y'})$. Since $d'=d-n<d$, by induction we have that $g\in\ext^i_\Gamma(\Delta_x,\Delta_{y'}\langle k \rangle^\sharp)$ for $k\ge 0$ such that $i=d'-nk$. Therefore we have
 
\[ f\in \Hom_\Gamma(\Delta_{y'},\Delta_y\langle 1 \rangle^\sharp)\cdot \ext^{d'-nk}_\Gamma(\Delta_x,\Delta_{y'}\langle k \rangle^\sharp)\subseteq\ext^{d'-nk}_\Gamma(\Delta_x,\Delta_{y}\langle k+1 \rangle^\sharp)\]
and we know that
\[\ext^{d'-nk}_\Gamma(\Delta_x,\Delta_{y}\langle k+1 \rangle^\sharp)=\ext^i_\Gamma(\Delta_x,\Delta_y\langle j \rangle^\sharp), \]
so $i=d'-nk=d-n-nk=d-n(k+1)$.
\end{itemize} \end{proof}

Denote by $\mathsf{tgr}{\Gamma^\dagger}$ the category of graded ${\Gamma^\dagger}$-modules with respect to the total grading.

\begin{lemma}\label{gamma morph}
If $Q_x\langle s\rangle\to Q_y$ is a non-zero morphism between indecomposable projective modules in $\mathsf{tgr}{\Gamma^\dagger}$, then $s=d-j(n-1)$ for some $j\ge 0$ and $d=d(x,y)$.
\end{lemma} 
\begin{proof} Any non-zero morphism $Q_x\langle s\rangle\to Q_y$ is given by left multiplication by an element in $\ext^i_\Gamma(\Delta_x,\Delta_{y}\langle j \rangle^\sharp)\subseteq\ext^*_\Gamma(\Delta_x,\Delta_{y})$ whose total grading is $d-j(n-1)$ by Lemma \ref{gamma totgrad}. Hence $s=d-j(n-1)$. \end{proof}

We are ready to prove the following:

\begin{thm} \label{dual is koszul}
Let $\Gamma$ be our qh-cover of an n-zigzag algebra with $n>1$, and let $\Gamma^\dagger=\ext^*_\Gamma(\Delta,\Delta)$. Then $\Gamma^\dagger$ endowed with the total grading $|\cdot|^{tot}$ is Koszul in the classical sense.   
\end{thm}
\begin{proof} 
We want to show that, for any two simple $\Gamma^\dagger$-modules $S_x,S_y$, if $\ext_{\mathsf{tgr}{\Gamma^\dagger}}^s(S_x,S_y\langle i,j\rangle)\neq 0$ then $i+j=s$. First recall that $S_x\cong G_\Delta(\nabla_x)$ by Proposition \ref{simple costand}. Moreover
\[ S_x\langle i,j\rangle\cong G_\Delta(\nabla_x)\langle i,j\rangle\cong G_\Delta(\nabla_x\langle-i,j\rangle[-i]) \]
by Proposition \ref{simple costand} and Proposition \ref{delta koszul propert}, (c).
Then we have:
\[\begin{split}\ext^s_{\mathsf{tgr}{\Gamma^\dagger}}(S_x,S_y\langle  i,j\rangle) & \cong \Hom_{\mathcal{D}(\mathsf{tgr}{\Gamma^\dagger})}(S_x,S_y\langle i,j\rangle[s]) \\ & \cong \Hom_{\mathcal{D}(\mathsf{tgr}\Gamma)}(\nabla_x,\nabla_y\langle -i,j\rangle[s-i]) \\ & \cong \ext^{s-i}_{\mathsf{tgr}\Gamma}(\nabla_x,\nabla_y\langle -i,j\rangle) \end{split}. \]

Recall also that from Proposition \ref{quasi hered}, an (ungraded) injective coresolution of $\nabla_y$ is:
\[0\to \nabla_y\to I_y \to I_{y_1} \to \cdots \to I_{y_k}=\nabla_z\to 0 \]
such that 
\[z\xto{\alpha_0} \cdots \xto{\alpha_0} y_1 \xto{\alpha_0} y \]
is a subquiver of the quiver of $\Gamma$.

Each (indecomposable) injective module in such a coresolution is cogenerated in $|\cdot|^\flat$-degree zero since applying the duality $D=\Hom_k(-,k)$ the corresponding maps between projective $\Gamma^{op}$-modules are given by right multiplication by $\alpha_0^*$ that are in degree zero. Moreover, for the same reason, we see that the coresolution is also linear with respect to $|\cdot|^\sharp$:
\[0\to \nabla_y\langle -i,j\rangle \to I_y\langle -i,j\rangle \to I_{y_1}\langle -i,j-1\rangle \to \cdots \to I_{y_k}\langle -i,j-k\rangle=\nabla_z\langle -i,j-k\rangle\to 0 \]

Hence if $\nabla_x\to I_{y_{s-i}}\langle -i,j-s+i\rangle$ is a non-zero map that gives a non-trivial element in the $\ext$-group, we must have $j=s-i$ since $\nabla_x$ is concentrated in $|\cdot|^\sharp$-degree zero. 

\end{proof}

\subsection{Presentation of \texorpdfstring{$\Gamma^\dagger$}{gamma dual} as bound quiver algebra} Let $(Q^{\Gamma^\dagger}_0,Q^{\Gamma^\dagger}_1)$ be the quiver of ${\Gamma^\dagger}$. Clearly the set of vertices $Q^{\Gamma^\dagger}_0$ is the same as the set of vertices of the quiver of $\Gamma$ and we will index this set always by $I$. We know that ${\Gamma^\dagger}$ is Koszul with respect to the total grading so ${\Gamma^\dagger}$ is generated by elements of degree one over its semisimple subalgebra in degree zero. Then the arrows of the quiver of ${\Gamma^\dagger}$ can be divided in two spaces: \begin{enumerate}
\item Arrows in $\ext$-degree zero: they correspond to the generators of $\Hom_\Gamma(\Delta_x,\Delta_y)$ whenever we have an arrow $y\xto{\alpha_0}x$ in the quiver of $\Gamma$. This $\Hom$-space is clearly one dimensional and gives us an arrow $a_0:y\to x$.

\item Arrows in $\ext$-degree one: they correspond to the generators of $\ext^1_\Gamma(\Delta_x,\Delta_y)$ that do not factor through a morphism as in point (1). Suppose that $\ext^1_\Gamma(\Delta_x,\Delta_y)\neq 0$ and let $P^\bullet\to\Delta_x$ be a (graded linear) projective resolution of $\Delta_x$ (note that $P^0=P_x$). The first $\ext$-space is the first cohomology group of the complex $\Hom_\Gamma(P^\bullet,\Delta_y)$ so its elements are equivalence classes of morphisms $P^1\to\Delta_y$. But since standard modules are concentrated in only one $\Delta$-degree and the differentials of $P^\bullet$ are in $\Delta$-degree one, the cohomology classes correspond to the $\Hom$-spaces. We know that $\Delta_y$ is either the simple module $S_y$ or, in case there exists an arrow $y\xto{\alpha_0}z$, it is uniserial with radical length two and composition factors $S_y(=\modtop\Delta_y)$ and $S_z(=\soc\Delta_y)$. If $\im(P^1\to\Delta_y)=S_z$ then the morphism factors through $\Delta_z$ hence is not irreducible; we have a (unique up to scalar multiplication) irreducible morphism $f:P^1\to\Delta_y$ if and only if $P^1=P'\oplus P_y$ and $f=0\oplus(P_y\onto\Delta_y)$. Therefore we have an arrow in $\ext$-degree one $a_i:y\to x\in Q_1^{\Gamma^\dagger}$ whenever there is an arrow $\alpha_i:x\to y\in Q^\Gamma_1$ for $i\neq 0$, since this happens if and only if the linear projective resolution of $\Delta_x$ is the following:
\[\cdots \to P^2 \to P^1=P'\oplus P_y \xto{[g,\alpha_i\cdot]} P_x\onto \Delta_x \]
for some $g$ in degree one.
\end{enumerate}

This means that $\ext^1_\Gamma(\Delta_x,\Delta_y)$ decomposes, as a $k$-vector space, in the direct sum of two at most one-dimensional vector spaces $\mathcal{U}\oplus\mathcal{V}$ where $\mathcal{U}=\ext_\Gamma^1(\Delta_x,\Delta_y)_0$ is generated by an irreducible morphism and $\mathcal{V}=\ext_\Gamma^1(\Delta_x,\Delta_y)_1$ is generated by a morphism that factors through a morphism between standard modules. 
\begin{exm}
Let $Z=Z^{(2,3)}$ and remember the quiver of $\Gamma$ from Example \ref{exm qhcover}. Then the quiver of ${\Gamma^\dagger}=\ext^*_\Gamma(\Delta,\Delta)$ is obtained by the quiver of $\Gamma$ by keeping the same arrows in $\Delta$-degree zero and reversing the arrows in $\Delta$-degree one:
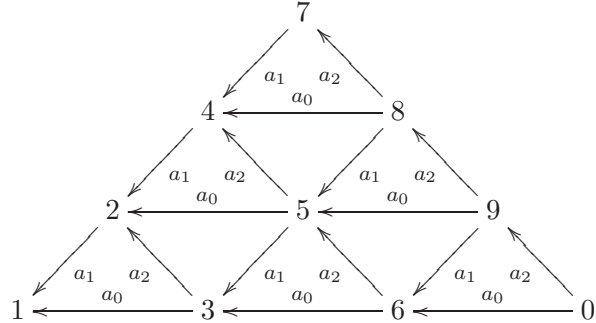
\[
\xymatrix{ & & & 7\ar[ld]^{a_1} & & & \\ 
               & & 4 \ar[ld]^{a_1}  & & 8\ar[ul]^{a_2}\ar[ld]^{a_1}\ar[ll]_{a_0} & &  \\
          & 2 \ar[ld]^{a_1} & & 5\ar[ul]^{a_2} \ar[ld]^{a_1}\ar[ll]_{a_0} & & 9\ar[ll]_{a_0}\ar[ld]^{a_1}\ar[ul]^{a_2} &  \\
          1  & & 3 \ar[ul]^{a_2}\ar[ll]_{a_0} & & 6 \ar[ul]^{a_2} \ar[ll]_{a_0} & & 0\ar[ll]_{a_0}\ar[ul]^{a_2}  }
\]
\end{exm}

By Theorem \ref{dual is koszul}, the ideal of relations of ${\Gamma^\dagger}$ is generated by homogeneous elements of degree two (with respect to the total grading). Then we need to find all the degree two relations of ${\Gamma^\dagger}$.

\begin{prop} \label{dual relations}
Let $(Q^{\Gamma^\dagger}_0,Q_1^{\Gamma^\dagger})$ be the quiver of ${\Gamma^\dagger}$.
\begin{enumerate}[(I)]
\item There are quadratic commutativity relations given by:    
      \begin{enumerate}[(a)]
      \item For any relation $\alpha_i\alpha_j=\alpha_j\alpha_i$ with $i,j\neq 0$ in $\Gamma$: $\vcenter{\xymatrix@=1em{ & y \ar[dr]^{\alpha_j} & \\
                       x \ar[ur]^{\alpha_i}\ar[dr]_{\alpha_j}\ar@{.}[rr] & & z\\
                        & w \ar[ur]_{\alpha_i} &  }}$, in ${\Gamma^\dagger}$ we have the relation $a_ia_j=a_ja_i$: $\vcenter{\xymatrix@=1em{ & y \ar[dl]_{a_i} & \\
                       x & & z \ar[ul]_{a_j}\ar[dl]^{a_i}\ar@{.}[ll]\\
                        & w \ar[ul]^{a_j} &  }}$.
       \item For any relation $\alpha_0\alpha_i=\alpha_i\alpha_0$ in $\Gamma$: $\vcenter{\xymatrix@=1em{ & y \ar[dr]^{\alpha_i} & \\
                       x \ar[ur]^{\alpha_0}\ar[dr]_{\alpha_i}\ar@{.}[rr] & & z\\
                        & w \ar[ur]_{\alpha_0} &  }}$, in ${\Gamma^\dagger}$ we have the relation $a_ia_0=a_0a_i$: $\vcenter{\xymatrix@=1em{ & y & \\
                       x\ar[ur]^{a_0} & & z \ar[ul]_{a_i} \\
                        & w \ar[ul]^{a_i} \ar[ur]_{a_0}\ar@{.}[uu] &  }}$.
      \end{enumerate}                  
\item There are quadratic monomial relations given by:
      \begin{enumerate}[(a)]
      \item $a_0a_0=0$
      \item $a_ia_j=0$ for any $i,j\neq 0$ such that $\alpha_i\alpha_j$ is not defined in the quiver of $\Gamma$.
      \end{enumerate}       
      
\end{enumerate}
\end{prop}
\begin{proof} 
For any $v\in I$ denote by $P^\bullet(v)$ a projective resolution of the standard module $\Delta_v$.
\begin{enumerate}[(I)]
  \item \begin{enumerate}[(a)]
          \item Consider the first three terms of a linear projective resolution of $\Delta_x$
          \[ \cdots\to P^2(x) \to \bigoplus_{x\xto{\alpha_i}x',i\neq 0}P_{x'}=P^1(x) \xto{f} P_x \]
In particular $P_y$ and $P_w$ are direct summands of $P^1(x)$ and the restriction of $f$ on these modules is $P_y\oplus P_w\xto{[\alpha_i\cdot,\alpha_j\cdot]}P_x$. The element $e_y\alpha_je_z-e_w\alpha_ie_z$ is such that $f(e_y\alpha_je_z-e_w\alpha_ie_z)=e_x\alpha_i\alpha_j e_z-e_x\alpha_j\alpha_i e_z=0$ hence it lies in the image of $P^2(x)\to P_y\oplus P_w$. But this is a linear map between projective modules, so we must have that $e_y\alpha_je_z-e_w\alpha_ie_z$ is in the image of $P_z\xto{[\alpha_j\cdot,\alpha_i\cdot]}P_y\oplus P_w$ and then $P_z$ must be a direct summand of $P^2(x)$. In particular $a_ia_j\in\ext^2_\Gamma(\Delta_x,\Delta_z)$ is non-zero.
      Moreover $a_ia_j$ and $a_ja_i$ are respectively represented by the following diagrams:
      
      \[ \xymatrix{ \cdots \ar[r] & P_z\oplus Q_2 \ar[r]^-{\left[\bsm \alpha_j\cdot \\ \alpha_i\cdot \esm\right]} \ar[d]^{\mbox{id}} & P_y\oplus P_w \oplus Q_1 \ar[d]^{\left[\bsm 1 & 0 & 0 \esm\right]} \ar[r]^-{\left[\bsm \alpha_i\cdot & \alpha_j\cdot & \ast \esm\right]} & P_x \ar[d]^0 \ar[r]^0 & 0 \\
                   \cdots \ar[r] & P_z\oplus Q'_2 \ar[d]^{\left[\bsm \mbox{id} & 0 \esm\right]} \ar[r]^-{\left[\bsm \alpha_j\cdot & \ast \\ \ast & \ast \esm\right] } & P_y \ar[d]^0 \ar[r]^0 & 0 & \\ \cdots \ar[r] & P_z \ar[r]^0  & 0 &  & } \]

and 

\[ \xymatrix{ \cdots \ar[r] & P_z\oplus Q_2 \ar[r]^-{\left[\bsm \alpha_j\cdot \\ \alpha_i\cdot \esm\right]} \ar[d]^{\mbox{id}} & P_y\oplus P_w \oplus Q_1 \ar[d]^{\left[\bsm 1 & 0 & 0 \esm\right]} \ar[r]^-{\left[\bsm \alpha_i\cdot & \alpha_j\cdot & \ast \esm\right]} & P_x \ar[d]^0 \ar[r]^0 & 0 \\
                   \cdots \ar[r] & P_z\oplus Q''_2 \ar[d]^-{\left[\bsm \mbox{id} & 0 \esm\right]} \ar[r]^{\left[\bsm \alpha_i\cdot & \ast \\ \ast & \ast \esm\right] } & P_w \ar[d]^0 \ar[r]^0 & 0 & \\ \cdots \ar[r] & P_z \ar[r]^0  & 0 &  & } \]
where $Q_1,Q_2,Q'_2$ and $Q''_2$ are projective $\Gamma$-modules such that $P_y\oplus P_w\oplus Q_1=P^1(x), P_z\oplus Q_2=P^2(x), P_z\oplus Q'_2=P^1(y)$ and $P_z\oplus Q''_2=P^1(w)$. Therefore we see that $a_ia_j=a_ja_i\in \ext^2_\Gamma(\Delta_x,\Delta_z)$.
         \item An element $a_ia_0\in \ext^1_\Gamma(\Delta_x,\Delta_w)\Hom_\Gamma(\Delta_y,\Delta_x)$, $i\neq 0$, is given by a diagram:
         
\[ \xymatrix{\cdots \ar[r] & P^2(y) \ar[r] \ar[d] & P_z\oplus Q \ar[r]^-{\left[\bsm \alpha_i\cdot & \ast \esm\right]}\ar[d]^{\left[\bsm \alpha_0\cdot & 0 \\ \ast & \ast \esm\right]} & P_y \ar[d]^{\alpha_0\cdot} \ar[r]^0 & 0 \\      \cdots \ar[r] & P^2(x)\ar[d] \ar[r] & P_w\oplus Q' \ar[d]^{\left[\bsm \mbox{id} & 0 \esm\right]} \ar[r]^-{\left[\bsm \alpha_i\cdot & \ast \esm\right]} & P_x \ar[d]^0 \ar[r]^0 & 0 \\ \cdots \ar[r] & P^2(w) \ar[r] & P_w \ar[r]^0 & 0 &  } \]

where $Q$ and $Q'$ are projective modules such that $P_z\oplus Q=P^1(y)$ and $P_w\oplus Q'=P^1(x)$. Note that the map $P_z\oplus Q \xto{\left[\bsm \alpha_0\cdot & 0 \\ \ast & \ast \esm\right]} P_w\oplus Q'$ has to be in $\Delta$-degree zero. Therefore the top-right entry of the corresponding matrix is zero since $P_z\to P_w$ is the unique $\Delta$-degree zero map with codomain $P_w$.

On the other hand the element $a_0a_i\in\Hom_\Gamma(\Delta_z,\Delta_w)\ext^1_\Gamma(\Delta_y,\Delta_z)$ is given by:

\[ \xymatrix{\cdots \ar[r] &  P_z\oplus Q \ar[r]^-{\left[\bsm \alpha_i\cdot & \ast \esm\right]}\ar[d]^{\left[\bsm \mbox{id} & 0 \esm\right]} & P_y \ar[d]^0 \ar[r]^0 & 0 \\      \cdots \ar[r] &  P_z \ar[d]^{\alpha_0\cdot} \ar[r]^0 & 0 \ar[d]^0  &  \\ \cdots \ar[r] & P_w \ar[r]^0 & 0 &  } \]

Hence we can conclude that $a_ia_0=a_0a_i$ since $\left[\bsm\mbox{id} & 0\esm\right]\left[\bsm \alpha_0\cdot & 0 \\ \ast & \ast \esm\right]=\left[\bsm\alpha_0\cdot & 0\esm\right]=\alpha_0 \left[\bsm\mbox{id} & 0\esm\right]$.
        \end{enumerate}
        
  \item \begin{enumerate}[(a)]
         \item Obvious, since by the structure of standard modules we have\\ $\Hom_\Gamma(\Delta_y,\Delta_z)\Hom_\Gamma(\Delta_x,\Delta_y)=0$ for any $x,y,z\in I$.
         \item Assume $i,j\neq 0$, so that $a_ia_j$ comes from some non-zero composition $x\xto{\alpha_j}y\xto{\alpha_i}z$ in the quiver of $\Gamma$ and we have no paths $\xymatrix{x\ar@{-->}[r]^{\alpha_i\alpha_j} & z}$. Let $\cdots\to P^2(z)\to P^1(z)\to P_z\onto\Delta_z$ be a graded linear projective resolution of $\Delta_z$; under our assumptions $P_x$ can not appear as a direct summand of $P^2(z)$ since the composition $P^2(z)\to P^1(z)\to P_z$ restricted to $P_x$ must coincide with $\alpha_j\alpha_i\cdot$ and so it would be non-zero. This means that $\ext^2_\Gamma(\Delta_z,\Delta_x)\cong\Hom_\Gamma(P^2(z),\Delta_x)=0$.

        \end{enumerate}
\end{enumerate} \end{proof}

\begin{thm} \label{dual algebra}
The algebra $\Gamma^\dagger$ is isomorphic to the path algebra of the following quiver:
%\begin{itemize}\renewcommand{\labelitemi}{$-$}
 \[Q_0^{\Gamma^\dagger}=Q_0^{\Gamma}\] \[\begin{split}Q_1^{\Gamma^\dagger}= & \left\lbrace x\xto{a_0}y : \mbox{ there exists } x\xto{\alpha_0}y \in Q_1^{\Gamma} \right\rbrace\cup \\ & \left\lbrace w\xto{a_i}z : \mbox{ there exists } z\xto{\alpha_i}w \in Q_1^{\Gamma}, i\neq 0 \right\rbrace \end{split}\] 
%\end{itemize}
bound by the ideal of relations $\mathcal{R}$ generated by elements as in Proposition \ref{dual relations}.
\end{thm}
\begin{proof}
We are left to prove that the relations of Proposition \ref{dual relations} are the only quadratic relations of $\Gamma^\dagger$. First of all note that, by the description of the quiver $Q^\Gamma$ given in Definition \ref{defin zigzag} and after reversing the arrows $\alpha_i$ with $i\neq 0$, given any two vertices $x,z\in Q_0^{\Gamma^\dagger}$ the $k$-basis of the vector space $z(kQ^{\Gamma^\dagger})_2x$ is either (1) the element $a_ia_i$ with $i\in \{0,\ldots ,n\}$, or (2) the set $\{a_ia_j, a_ja_i\}$ or (3) the element $a_ia_j$ with $i,j\in \{0,\ldots ,n\}$, $i\neq j$, such that the composition $a_ja_i$ is not defined in the quiver. Case (2) occurs precisely when the path $a_ia_j$ is part of a mesh $\vcenter{\xymatrix@=1em{ & y \ar[dl]_{a_i} & \\ x & & z \ar[ul]_{a_j}\ar[dl]^{a_i}\ar@{.}[ll] \\  & w \ar[ul]^{a_j} &  }}$ for $i,j\in \{0,1,\cdots,n\}$, as described in Proposition \ref{dual relations}. Let us discuss the possible relations in these three cases:
\begin{itemize}
\item[(1)] The elements $a_0a_0$ are zero in $\Gamma^\dagger$ by Proposition \ref{dual relations} (II)(a). The elements $a_ia_i$ for $i\neq 0$ are non-zero since the same argument used in Proposition \ref{dual relations} (I)(a) for $i=j$ shows an explicit extension.

\item[(2)] Any element $a_ia_j$ that is part of a square is subject to the commutativity relation $a_ia_j=a_ja_i$ by Proposition \ref{dual relations} (I)(a,b). Moreover such an element is non-zero since, as before, an explicit extension is given in the proof of Proposition \ref{dual relations}.% Hence $z(kQ^{\Gamma^\dagger})_2x$ is one-dimensional.

\item[(3)] Let $a_ia_j$ be a path of length two and suppose that $i=0$ and $j\neq 0$. Then any composition $a_0a_j$ in the quiver of ${\Gamma^\dagger}$ comes from two arrows $\vcenter{\xymatrix@=1em{ y \ar[dr]^{\alpha_j} & \\
                        & z\\
                        w \ar[ur]_{\alpha_0} &  }}$ in the quiver of $\Gamma$. But any such couple of arrows is part of a square $\vcenter{\xymatrix@=1em{ & y \ar[dr]^{\alpha_j} & \\
                       x \ar[ur]^{\alpha_0}\ar[dr]_{\alpha_j}\ar@{.}[rr] & & z\\
                        & w \ar[ur]_{\alpha_0} &  }}$ where $\alpha_0\alpha_j\neq 0$ is defined. Moreover the corresponding square in $Q^{\Gamma^\dagger}$ is commutative by Proposition \ref{dual relations} (I)(b) . The case for $j=0$ and $i\neq 0$ is similar.
Hence, for elements $a_ia_j$ such that $a_ja_i$ is not part of the quiver, we can always assume $i,j\neq 0$ and these elements are zero by Proposition \ref{dual relations} (II)(b). 
\end{itemize}
We have shown all the possible quadratic relations so the proof is complete.
\end{proof}
%\newpage

%\bibliography{bibliogabri.bib}{}
%\bibliographystyle{alpha}

\end{document}